\documentclass[12pt]{amsart}
\usepackage{amssymb}
\usepackage{amsmath} 
\usepackage{amsthm} 
\usepackage{epsfig}
\usepackage{graphicx}
\usepackage{color}
\usepackage{transparent}
\usepackage{soul}
\usepackage{subfig}
\usepackage{wrapfig}

\newcommand{\re}{\mathrm{Re}\,}

\newtheorem{thm}{Theorem}[section]
\newtheorem{lem}[thm]{Lemma}

\newtheorem{cor}[thm]{Corollary}
\newtheorem{rem}[thm]{Remark}

\newtheorem{prop}[thm]{Proposition}

\newcounter{probno}
\stepcounter{probno}
\newenvironment{prob}{

\vspace*{.2in}
\noindent{\bf Problem \arabic{probno}.}}
{\stepcounter{probno}
\vspace*{.2in}

}

\renewcommand{\qed}{\hfill$\Box$ \par \medskip}

\newcommand{\id}{\mathrm{id}}

\newcommand{\pair}[2]{\left\langle #1,#2 \right\rangle}

\newcommand{\R}{\mathbf{R}}

\newcommand{\C}{\mathbf{C}}
\newcommand{\cp}{\mathbf{P}}
\newcommand{\Z}{\mathbf{Z}}
\newcommand{\N}{\mathbf{N}}

\newcommand{\Q}{\mathbf{Q}}

\newcommand{\tto}{\dashrightarrow}
\newcommand{\ve}{\mathbf{e}}

\newcommand{\aut}{\mathop{\mathrm{Aut}}}
\newcommand{\crit}{\mathop{\mathrm{Crit}}}
\newcommand{\htop}{h_{top}}

\newtheorem*{thma}{Theorem A}
\newcommand{\maximal}{{A}}
\newtheorem*{thmb}{Theorem B}
\newcommand{\allreal}{{C}}
\newtheorem*{thmc}{Theorem C}
\newcommand{\nonmaximal}{{B}}
\newcommand{\qu}[1]{#1}
\newcommand{\cg}[1]{#1}
\newcommand{\chg}[1]{\textcolor{black}{#1}}

\newcommand{\fp}{p_{fix}}

\title{Entropy of real rational surface automorphisms}

\author{Jeffrey Diller}
\address{Department of Mathematics\\
         University of Notre Dame\\
         Notre Dame, IN 46556}
\email{diller.1@nd.edu}

\author{Kyounghee Kim}
\address{Department of Mathematics\\
         Florida State University\\
         Tallahassee, FL 32308}
\email{kim@math.fsu.edu}

\subjclass{}
\keywords{}

\begin{document}

\begin{abstract}
We compare real and complex dynamics for automorphisms of rational surfaces that are obtained by lifting \chg{some} quadratic birational maps of the plane.  In particular we show how to exploit the existence of an invariant cubic curve to understand how the real part of an automorphism acts on homology.  We apply this understanding to give examples where the entropy of the full (complex) automorphism is the same as its real restriction.  Conversely and by different methods, we exhibit different examples where the entropy is strictly decreased by restricting to the real part of the surface. \cg{Finally, we give an example of a rational surface automorphism with positive entropy whose periodic cycles are all real.} 
\end{abstract}

\maketitle

\section*{Introduction}
Any automorphism of the complex projective plane $\cp^2$ is \cg{ a linear projective transformation}.  If one blows up sufficiently many points in $\cp^2$, however, the resulting rational surface $X\to\cp^2$ sometimes admits an automorphism $f:X\to X$ with much more interesting dynamics.  This was known already in some sense by Coble \cite{cob}, but recent papers (\cite{beki}, \cg{\cite{bla},} \cg{\cite[Exemple 9.4]{can3}}, \cite{mcm}, etc) furnish new constructions and many more examples.  Many of these automorphisms are `real', in the sense that the points blown up lie in $\cg{\cp^2(\R)}\subset\cp^2$ and the automorphism \chg{$f=f_\C:X\to X$} restricts to a diffeomorphism $f_\R:X(\R)\to X(\R)$ of the real two dimensional submanifold of $X$ lying over $\cg{\cp^2(\R)}$.  

Our purpose in this article is to compare the real and the complex dynamics of such automorphisms.  Since dynamical complexity of a diffeomorphism $f$ is usually quantified by its topological entropy $h_{top}(f)$, we consider in particular whether \chg{$h_{top}(f_\C) = h_{top}(f_\R)$} for a real automorphism $f:X\to X$ on a blowup $X$ of $\cp^2$.  When this happens we say that $f_\R$ `has maximal entropy.'  While our results are limited to particular families of examples, they indicate a range of possibilities, give reasonable methods for verifying or disproving equality of entropy, and raise some interesting questions for further investigation.

Any quadratic plane birational map has at most three distinct indeterminate points.  Here we restrict attention to those which are non-degenerate in the sense that they have exactly three such points.  That is, we \chg{consider quadratic birational maps of the} form $\check f := T\circ J$, where $T\in\aut(\cp^2)$ is linear and $J[x_1,x_2,x_3]=[x_2x_3,x_3x_1,x_1x_2]$ is the standard quadratic involution, presented in homogeneous coordinates.  Because of the second factor, $\check f$ contracts each of the lines $E_j := \{x_j=0\}$, $j=1,2,3$, to a distinct point in $\cp^2$, and $\check f$ is indeterminate at the points $p_1 := [1,0,0], p_2 := [0,1,0], p_3 := [0,0,1]$ where the lines meet pairwise.  Suppose that $T$ is somehow chosen so that there are positive integers $n_1,n_2,n_3$ and a permutation $\sigma\in\Sigma_3$ for which all points $\check f^j(E_i)$, $1\leq j \leq n_i$, are distinct and 
\begin{equation}
\label{eqn:orbitrelation}
\check f^{n_j}(E_j) = p_{\sigma(j)}, \quad j=1,2,3.
\end{equation}
Then if $\pi:X\to\cp^2$ is the blowup of $\cp^2$ at all points $\check f^k(E_j)$, $j=1,2,3$ and $1\leq k\leq n_j$ in the three critical orbits, the birational map $\check f$ lifts to an automorphism $f:X\to X$.  We call $n_1,n_2,n_3,\sigma$ \emph{orbit data} and say that the data is \emph{realizable} if the linear map $T$ can be chosen so that \eqref{eqn:orbitrelation} holds.  

Note that if the linear map $T$ is real then so are the critical orbits and the map $f$.  In particular $f$ restricts to a diffeomorphism $f_\R:X(\R)\to X(\R)$ of the compact real surface $X(\R) := \overline{\cg{\cp^2(\R)}\setminus \crit(\pi)} \subset X$.  Our first result is

\begin{thma} One can choose a real linear $T\in\aut(\cp^2)$ so that the map $\check f = T\circ J$ lifts to an automorphism $f:X\to X$ such that $h_{top}(f_\R) = h_{top}(f) >0$.  This happens in particular for \chg{some $T$ so that} $\check f$ realizes one of the following sets of orbit data:
\begin{itemize}
 \item $\sigma = (123)$ is cyclic, $n_1=n_2 =1$ and $n_3 \geq 8$;
 \item $\sigma = (123)$ is cyclic, $n_1=2$, $n_2=n_3\geq 4$.
\end{itemize}
\end{thma}

For the first set of orbit data, this theorem was established by Bedford and the second author in \cite{beki2}.  The methods there were somewhat ad hoc.  Our proof here is more systematic and relies on two important inequalities that bound the entropy of a map in terms of the induced pushforward action on homology.  
Specifically, we use the inequalities
$$
\log\rho(f_{\R,*})\leq h_{top}(f_\R) \leq h_{top}(f) = \log \rho(f_*),
$$
where $\rho$ denotes spectral radius and $\cg{(f_\R)_*}, f_*$ are the induced linear actions on $H_1(X(\R);\R)$ and $H_2(X;\R)$, respectively.  The first and last (in)equalities are due to Yomdin \cite{yom} \qu{(see also \cite{man})} and Gromov \cite{gro}.  Together, they imply that $f_\R$ will have maximal entropy when $f_{\R}$ expands homology classes of real curves as fast as $f_*$ expands homology classes of complex curves.  

It is not difficult to write $f_*$ down explicitly in the present context, but sign issues make it more difficult to find $\cg{(f_\R)_*}$.  There is no natural orientation for closed curves in $X(\R)$ that will be respected by $\cg{(f_\R)_*}$.  To cope with this we impose an additional condition on $\check f$ that has proven important in earlier work (\cite{mcm,dil,ueh}) for guaranteeing that \eqref{eqn:orbitrelation} is satisfied.  Namely, we require that all critical orbits of $\check f$ lie on an $\check f$-invariant cubic curve $C$.  This additional condition allows us to resolve orientation issues and thereby effectively compute $\cg{(f_\R)_*}$.

While the work devolves into analysis of various cases, we are finally able to compute the action $\cg{(f_\R)_*}$ for automorphisms arising from essentially any orbit data that can be realized by a non-degenerate quadratic birational map that fixes the curve $C$.  More often than not, it turns out that $\rho(\cg{(f_\R)_*}) < \rho(f_*)$ so that the above method does not tell us whether or not $f_{\R}$ has maximal entropy.  In fact, in some cases $\rho(f_*) > 1$ while the real action $\cg{(f_\R)_*}$ is periodic.  That is, $h_{top}(f)$ is positive, whereas since $\rho(\cg{(f_\R)_*}) = 1$, we cannot infer the same for $h_{top}(f_\R)$.  We do not know whether $h_{top}(f_\R)$ actually does vanish in any particular case, but we can at least show that $h_{top}(f_\R)$ sometimes fails to be maximal.

\begin{thmb}
One can choose a real linear $T\in\aut(\cp^2)$ so that $\check f = T\circ J$ lifts to an automorphism $f:X\to X$ such that $h_{top}(f) > h_{top}(f_\R)$.  This happens in particular when $T$ is chosen so that $\check f$ realizes the orbit data $\sigma = (123)$, $n_1 = n_2 = 3$, $n_3>>1$.
\end{thmb}

Our proof of this theorem relies on a different set of ideas.  On the one hand, when $h_{top}(f) > 0$, it is known from work of Bedford-Lyubich-Smillie \cite{belysm} and Cantat \cite{can1} that there is a unique (necessarily ergodic) $f$-invariant measure $\mu$ with maximal metric entropy $h_\mu(f) = h_{top}(f)$.  It is also known \cite{can2} that the support of $\mu$ includes every saddle periodic point of $f$ whose stable and unstable manifolds are not algebraic.  On the other hand, Newhouse \cite{new} showed that every real surface diffeomorphism admits \emph{some} invariant measure of maximal entropy.

To prove our second theorem we show that the given orbit data is realized by a map $f:X\to X$ with two saddle fixed points outside \cg{the set of real points $X(\R)$}.  The nature of the fixed points is established by direct computation based on an explicit formula for $f_\R$.  In joint work with Bedford \cite{bediki}, we explained how to find such a formula.  Again, our construction of $f$ produces an underlying birational map $\check f$ with a (real) invariant cubic curve.  The presence of the invariant cubic essentially rules out other invariant curves, so that in particular the stable and unstable manifolds of the complex saddle points are transcendental.  Since these fixed points must lie in the support of $\mu$, any measure of maximal entropy for $f_\R$ must differ from $\mu$.  Uniqueness of $\mu$ further implies that the entropies of $f_\R$ and $f$ differ.

\cg{
Our third result, not included in an earlier version of this paper, answers a question posed by the referee.
\begin{thmc}
There exists a rational surface automorphism $f:X\to X$ with positive entropy such that all periodic cycles of $f$ lie in the real locus $X(\R)\subset X$.
\end{thmc}
We prove this for the automorphism $f$ associated to orbit data $2,4,5,\id$.  As in the cases considered in Theorem { \maximal}, the restriction $f_\R:X(\R)\to X(\R)$ of $f$ has maximum possible homology growth and therefore maximal entropy.  The general strategy is to compare the counts of real and complex periodic points of $f$ by applying the Lefschetz fixed point formula.  The count of real points is more complicated, because $X(\R)$ is non-orientable, which leads us to work on the orientation double cover $\hat\pi:\hat X \to X(\R)$ rather than directly on $X(\R)$.
}

Sections \ref{sec:homology} and \ref{sec:basic} of this paper present some necessary background about entropy and homology and then about the quadratic birational maps and associated surface automorphisms of interest here.  Section \ref{sec:action} shows how to compute the action $\cg{(f_\R)_*}:H_1(X(\R);\R)\to H_1(X(\R);\R)$ for such an automorphism when it preserves a cuspidal cubic $C$.  The proof of Theorem {\maximal } is given here.  Section \ref{sec:saddle} gives the proof of Theorem \nonmaximal, and Section \ref{sec:reallyreal} contains the proof of Theorem \allreal.

\chg{Clearly there are many interesting questions left open.  Not only is the set of automorphisms we consider very limited, even the underlying real surfaces we allow are rather restricted, omitting not only irrational surfaces but even many rational surfaces with a natural real structure, e.g. surfaces obtained by blowing up pairwise conugate complex points in $\cp^2(\C)$.  We note in particular some interesting work of Moncet \cite{mon1}, involving a very different set of ideas, that allows him to compare real vs. complex entropy for real automorphisms of various irrational surfaces.}  

There is more to be understand even for the examples we consider here.  It is natural, for instance, to seek a more detailed description of the dynamics of $f_\R$ when $\rho((f_\R)_*) = \rho(f_*)$.  And in the other direction, we would like to know whether both inequalities in the chain $\log\rho((f_\R)_*) \leq h_{top}(f_\R) \leq \log\rho(f_*)$ must be strict whenever one of them is.  We discuss these and other interesting open problems at greater length in the concluding Section \ref{sec:conclusion}.  The appendix gives an exhaustive summary of our computation of $(f_\R)_*$ for arbitrary sets of orbit data.

\qu{We thank the referee for an especially thorough reading and many constructive suggestions for improvements.  They have greatly improved this paper.  We thank Eric Bedford for offering his ideas about counting periodic points for maximal entropy automorphisms. These were essential to us for proving Theorem { \allreal}.}

\section{Entropy and homology}
\label{sec:homology}

In this section and the next, we give some necessary background.  \qu{It is standard to quantify the dynamical complexity of a continuous map $F:M\to M$ on a compact metric space $M$ in terms of its topological entropy $\htop(F)$.  Readers unfamiliar with entropy can consult \cite{kaha} for a precise definition and a thorough development.  For our purposes, however, it will suffice to recall some important and deep connections between the entropy and homology.  The first is due to Yomdin \cite{yom} (see \cite{man} for the case of first homology).}

\begin{thm}
\label{thm:yomdin}
\cg{Let $F:M\to M$ be a $C^\infty$-smooth self-map of a compact connected differentiable manifold.} Let $\rho(F_*)$ denote the spectral radius of the pushforward action $F_*:H_*(M;\R)\to H_*(M;\R)$ on the total real homology of $M$.  Then
$$
\htop(F) \geq \log \rho(F_*),
$$
\end{thm}

When $M$ is a real surface and $F$ is a diffeomorphism, we have that $F_* = \id$ on $H_0(M;\R)$ and $H_2(M;\R)$.  Hence $\rho(F_*) = \rho(F_*|_{H_1(M;\R)})$.  

The following complementary result of Gromov \cite{gro} tells us that \qu{in the K\"ahler setting} the inequality in Theorem \ref{thm:yomdin} is actually an equality.

\begin{thm}
\label{thm:gromov}
Let $F:M\to M$ be a holomorphic self-map of a compact K\"ahler manifold.  Then 
$$
h_{top}(F) \leq \log\max_k \rho(F_*|_{H_{2k}(M;\R)}). 
$$
\end{thm}

\noindent When $M$ is a \qu{compact K\"ahler surface (i.e. $\dim_\C M = 2$)} and $F$ is an automorphism, we have that $F_* = \id$ on $H_0(M;\R)$ and $H_4(M;\R)$.  Hence $\rho(F_*) = \rho(F_*|_{H_2(M;\R)})$; i.e. it suffices to consider only $k=1$ in Theorem \ref{thm:gromov}.

As explained in the introduction, we are interested in $M$ equal to either $X$ or $X(\R)$, where $\pi:X\to\cp^2$ is the blow up of finitely many distinct real points $p_1,\dots ,p_N\in \cg{\cp^2(\R)} \subset \cp^2$, and 
$$
X(\R) := \overline{\pi^{-1}(\cg{\cp^2(\R)}\setminus\{p_1,\dots,p_N\})}
$$ is the set of real points of $X$.  In this case $H_2(X;\Z)$ (and therefore also $H_2(X;\R)$) is generated by the homology classes of the exceptional curves $E_i := \pi^{-1}(p_i)$ together with the class of $\pi^{-1}(L)$, where $L\subset \cp^2$ is any line disjoint from the points $p_i$.

Since each $p_i \in\cg{\cp^2(\R)}$, we have that $E_i\cap X(\R)$ is a smooth circle $e_i$, the projectivization of the real tangent space $T_{p_i} X_{j-1,\R}$ at $p_i$.  Again the homology group \chg{$H_1(X(\R);\Z) = \Z^N \oplus (\Z/2\Z)$} is generated by the circles $e_i$ and the class of a \cg{generic} real line $\ell$.  Indeed the $e_i$ generate the free part of the homology and therefore all of $H_1(X(\R);\R)$.  However, in contrast with the complex situation in which both $X$ and the exceptional curves $E_i$, carry canonical orientations, the real surface $X(\R)$ is non-orientable, and there is no natural way to orient the exceptional circles $e_i$ within $X(\R)$.  \qu{Moreover, different choices of generic real line $\ell$ can give rise to different homology classes (see Proposition \ref{prop:classofrealline} below)}.  

We impose the further condition, whose purpose becomes clearer in the next section, that all $p_j$ are regular points on the cuspidal cubic curve 
\begin{equation}
\label{eqn:cuspcubic}
C := \{[x,y,z]\in\cp^2: yz^2 = x^3\}.
\end{equation}  
If we identify $\{z=0\}$ with the line at infinity in $\cp^2$ and let $(x,y)\in\cg{\mathbb{A}^2}\mapsto [x,y,1]$ denote affine coordinates on the complement, then the cusp $[0,1,0]$ is the unique point at infinity for $C$, and the regular part $C_{reg}:=C\cap\cg{\mathbb{A}^2}$ of $C$ is parametrized by $\gamma(x) = (x,x^3)$.  We identify $C$ and its real slice $\qu{C(\R) := C\cap \cp^2(\R)}$ with their strict transforms in $X$ and $X(\R)$, respectively.  The parametrization $\gamma$ provides an orientation for $\qu{C(\R)}$, and it is convenient to give all circles $e_j$ the `clockwise' orientation relative to affine coordinates $(x,y)$.  

Let $\ell\subset \cg{\cp^2(\R)}$ be a real line in $\cg{\cp^2(\R)}$, and \cg{suppose $\ell$ is not the line at infinity}.  If $\ell$ is vertical we orient $\ell$ in the downward direction, and if not, we orient $\ell$ from left to right.  As with complex curves in $\cp^2$, we implicitly identify the real curve $\ell$ with its strict transform $\overline{\ell\setminus (p_j)_{\{1\leq j \leq N\}}}$ in $X(\R)$.  
As a curve in $\cg{\cp^2(\R)}$, the line $\ell$ meets $C$ in either $1$ or $3$ points, counting with multiplicity.
The next fact gives us a very convenient description of the homology class of $\ell$ in $X(\R)$ in terms of the locations of these intersections.  We will say that $p\prec q$ for two points \qu{$p,q\in C(\R)$} if $p = \gamma(x),q = \gamma(x')$ with $x<x'$.  

\begin{figure}
\centering
\def\svgwidth{6.5truein}
  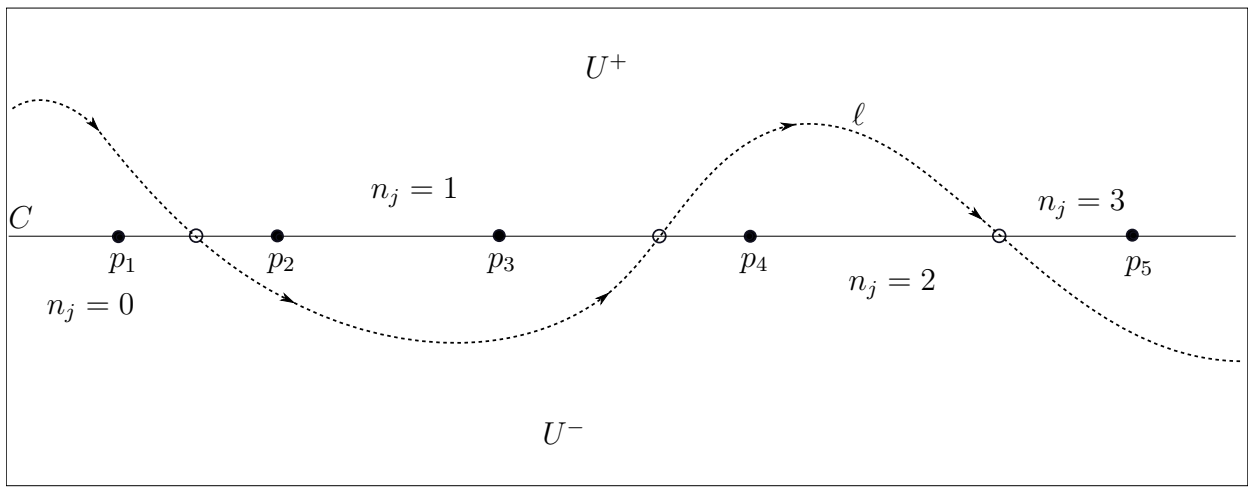
\caption{Homology class of a finite real line. The center horizontal line represents the cuspidal cubic $C$ and the dashed curve represents a real line $\ell$ that meets $C$ in three distinct points. Filled circles are points blown up to create the rational surface $X$. \label{fig:homologyLine}}
\end{figure}

\begin{prop}
\label{prop:classofrealline}
Let $\ell\subset\cg{\cp^2(\R)}$ be any real line which is distinct from the line at infinity.  The homology class of $\ell$ in $H_1(X(\R);\Q)$ is given by 
$$
\ell \sim  \sum_{p_j\notin \ell} (-1)^{n_j} e_j,
$$
where $n_j = \#\{p \in \ell\cap C: p\prec p_j\}$.
\end{prop}

Figure \ref{fig:homologyLine} illustrates this proposition in a caricatured style we will use throughout this article. For emphasis and the sake of keeping things visually separated, we will always draw the cubic $C(\R)$ as horizontal line in affine space $\mathbf{A}^2(\R)$.  Any real line different from the line at infinity will appear as a dashed curve intersecting $C(\R)$ at various points.  \chg{In Figure \ref{fig:homologyLine}, $\ell\cap C$ consists of three (real) points.  None of these precedes $p_1$, one precedes $p_2$ and $p_3$, two precede $p_4$, and all three precede $p_5$.  Hence $\ell  \sim  e_1-e_2-e_3+e_4-e_5$. }

\begin{proof}
First note that any vertical line in $\ell\subset\R^2$ is homotopic in $X(\R)$ to nearby lines $\ell'$ with negative slope and the same single point of intersection with $C$.  So we can assume without loss of generality that $\ell$ is not vertical. Next observe that removing intersections with $\ell$ decompose the regular (i.e. finite) part of \qu{$C(\R)$} into two or four connected components (the second or third will be empty if $\ell$ is tangent to \qu{$C(\R)$} at some point), the leftmost below $\ell$, the second above $\ell$ and so on.  From these observations, one sees that the proposition is equivalent to showing that
\begin{equation}\label{eqn:abovebelow}
\ell \sim \sum_{p_i \text{ below } \ell} e_i - \sum_{p_j \text{ above } \ell} e_j.
\end{equation}
To establish this last statement, \chg{assume for the moment that $\ell$ does not contain any of the points $p_j$}.  Let $\ell_\infty\subset\cg{\cp^2(\R)}$ denote the line at infinity, oriented so that the `upper' connected component $U^+\subset \mathbb{A}^2(\R)\setminus\ell$ has boundary $\partial U^+ = \ell_\infty -\ell$.  \cg{For each $p_j \in U^+$, let $ u_j \subset U^+$ be a small clockwise circle centered at $p_j$. Let $\tilde U^+$ denote the region obtained from $U^+$ by removing the disks bounded by the $u_j$.  Then $\partial \tilde U^+ = \ell_\infty - \ell - \sum_{p_j\in U^+} u_j$.

After we blow up all the $p_j$, the open set $U^+$ lifts to an open set $\tilde U$ containing exceptional circles $e_j$ in place of the points $p_j$, and each circle $u_j$ is homotopic to $2e_j$. It follows that $\partial \tilde U^+ \sim \ell_\infty - \ell - 2\sum_{p_j \in U^+} e_j \sim 0$.  }
Therefore in $X(\R)$, we have that
$$
\ell_\infty - \ell \sim 2\sum_{p_j \in U^+} e_j
$$
We also have that $2\ell_\infty$ is null-homotopic in $\cg{\cp^2(\R)}$.  Hence $2\ell_\infty$ is homotopic to $2\sum e_j$ in $X(\R)$.  Putting these relations together proves \eqref{eqn:abovebelow}.

\chg{
If $\ell$ contains a point $p_j$, one argues similarly. The only difference is that, $\ell$ bisects the circle $u_j$.  Hence one of the connected components of $\partial \tilde U^+$ consists of $\ell$ and the \emph{upper} half of $b_j$, both with coefficient $-1$.  As the radius of $b_j$ tends to zero, the upper half of $b_j$ shrinks in $X(\R)$ to a \emph{single} copy of $e_j$.  Hence in the limit 
$$
0\sim \partial \tilde U^+ \sim \ell_\infty - \ell - \sum_{p_j \in U^+} 2e_j - \sum_{p_j\in \ell} e_j.
$$
Using the further relation $\ell_\infty \sim \sum e_j$ to eliminate $\ell_\infty$ also eliminates all $e_j$ in the second sum.}
\end{proof}

\section{Automorphisms from quadratic birational maps}
\label{sec:basic}

As indicated in the introduction, our focus will be on automorphisms $f:X\to X$ on blowups $\pi:X\to\cp^2$ associated to certain quadratic birational maps of $\cp^2$.  Such automorphisms are explored at length in \cite{dil} and \cite{bediki}.  Let us summarize the notation and results that we need from those articles.

Let $\check f:\cp^2\tto\cp^2$ be a quadratic plane birational map, \qu{non-degenerate in the sense that it is indeterminate at three distinct points in $\cp^2$}.  That is,
\begin{equation}
\label{basicmap}
\check f = T^-\circ J \circ (T^+)^{-1},
\end{equation}
where $T^\pm\in\aut(\cp^2)$ are linear and $J$ denotes the \cg{ standard quadratic }involution given in affine coordinates by $(x,y)\mapsto (x^{-1},y^{-1})$.  \chg{Such $f$ are Zariski dense in the set of all quadratic birational maps.}

The map $J$ acts on the triangle with vertices $\ve_1 := [1,0,0], \ve_2:=[0,1,0], \ve_3:= [0,0,1]$ by contracting each side to the opposing vertex.  \chg{However, $J$ is indeterminate (i.e. not continuously defined) at the vertices themselves, instead blowing up each} to the opposite side of the triangle.  Away from the triangle, $J$ is a diffeomorphism.  Now let $p_i^\pm := T^\pm(\ve_i)$ and $E_i^\pm$ denote the vertices and opposite sides, respectively, of the images of the critical triangle for $J$ by $T^+$ and $T^-$.  Then $E_i^+$ is critical for $\check f$ with $\check f(E_i^+) = p_i^-$, and $p_i^+$ is indeterminate for $\check f$ with $\check f(p_i^+) = E_i^-$.

Note that here and elsewhere, we employ the convention that for any curve $V\subset\cp^2$, the image $\check f(V) := \overline{\check f(V\setminus \{p_1^+,p_2^+,p_3^+\})}$ omits the images of any indeterminate points on $V$; i.e. $\check f(V)$ is the set theoretic `strict transform' of $V$.

We will always assume that the linear maps $T^\pm$ defining $\check f$ are real, i.e. that they restrict to automorphisms of $\cg{\cp^2(\R)}$.  For purposes of this paper, we call such an $\check f$ a \emph{basic real map}.  Note that $\check f$ can be linearly conjugated to a map of the form $T\circ J$ (where $T= (T^+)^{-1}\circ T^-$) used in the introduction.  

Let $C$ be the cuspidal cubic curve \qu{defined by \eqref{eqn:cuspcubic} above}.  We will say that a birational map $f$ \emph{properly fixes} $C$ if $\check f(C) = C$ and none of the indeterminate points $p_i^\pm$ are the cusp of $C$.  Our interest in $C$ begins with the fact that there are many basic real maps that properly fix $C$, and that there is an easy way to characterize them.

\begin{prop}[\cg{\cite[Theorem~1.3]{dil}}]
Let $g:\cp^2\tto\cp^2$ be a \qu{quadratic} birational map.  Then $g$ properly fixes $C$ if and only if all indeterminate points of $g$ and $g^{-1}$ are contained in $C_{reg}$.  
\end{prop}

When $\check f$ properly fixes $C$, it necessarily fixes the cusp of $C$ and therefore also the complement $C_{reg} = \gamma(\C)$.  Thus $\check f$ restricts to an automorphism $f_C:C_{reg}\to C_{reg}$ given by $f_C(\gamma(x)) = \gamma(\delta x + \tau)$ for constants $\delta \neq 0$ and $\tau$.  When $\check f$ is real, then of course $\delta,\tau\in \R$ and \qu{$\check f(C(\R)) = C(\R)$}.  Following \cite{mcm}, we call $\delta$ the \emph{determinant} of $\check f$.  This is because \chg{$\check f^*\omega = \delta\omega$} where $\omega$ is the unique (up to constant multiple) meromorphic two form on $X$ with divisor equal to $-C$.

\begin{prop}
\label{prop:realizable}
Suppose $\check f$ properly fixes $C$ and that there exist integers $n_1,n_2,n_3 >0$ and a permutation $\sigma\in \Sigma_3$ such that 
\begin{itemize}
\item for all $1\leq j< n_i$, the point $p_{i,j} := \check f^j(E_i^+) = f_C^{j-1}(p_i^-)$ is not indeterminate for $\check f$;
\item $p_{i,n_i} := \check f^{n_i}(E_i^+) = p_{\sigma (i)}^+$;
\end{itemize}
Then $\check f$ lifts to an automorphism $f:X\to X$, where $\pi:X\to\cp^2$ is the blowup of $\cp^2$ at the points $p_{i,j} := \check f^j(E_i^+)\in C$, $1\leq k \leq n_j$, $j=1,2,3$.
\end{prop}

Note that the hypotheses of the proposition imply that the points $p_{i,j}$ are all distinct.
We refer to the positive integers $n_1,n_2,n_3$ and permutation $\sigma$ collectively as \emph{orbit data} and say that the map $\check f$ in Proposition \ref{prop:realizable} realizes this orbit data.  From the orbit data, one easily determines the action $f_*:H_2(X;\R) \to H_2(X;\R)$.  Specifically, if $E_{i,j}\subset X$ denotes the exceptional curve obtained by blowing up the point $p_{i,j}$, then $f_*$ acts by 
\begin{equation}
\label{h2action}
E^+_i \mapsto E_{i,1}\mapsto\dots\mapsto E_{i,n_i} \mapsto E^-_{\sigma(i)},\quad{and}\quad L\mapsto 2L - \sum_{i=1}^3 E_{i,1},
\end{equation}
where $L$ is (the homology class in $X$ of) a general line in $\cp^2$.  This suffices for describing $f_*$ 
\qu{because of the further homology relation $E_i^- \sim L - \sum_{j\neq i} E_{j,1}$.  This last relation expresses (for $f^{-1}$) the fact that each exceptional line of a basic real map contains two of the map's three indeterminate points}.

Because the operator $f_*$ depends only on the orbit data realized by $\check f$, it makes sense to talk about the characteristic polynomial associated to any given orbit data regardless of whether or not it is realized by some basic map.  The characteristic polynomial \cite{difa} has at most two, necessarily real, roots $\delta>1>\delta^{-1}$ outside the unit circle. If these roots exist, then $n_1+n_2+n_3 \geq 10$.

\begin{thm}[\cite{mcm,dil}]
\label{thm:realizable}
Let $\check f$ be a basic real map satisfying the hypothesis of Proposition \ref{prop:realizable}.  Then the determinant $\delta$ of $\check f$ is a real root of the characteristic polynomial of $f_*:H_2(X,\R)\to H_2(X,\R)$.  For real $\delta \neq \pm 1$, there is an affine change of parameter $\tilde \gamma(t) = \gamma(\alpha t + \beta$) such that $p_{fix}:=\tilde\gamma(0) \in C\cap \cg{\mathbb{A}^2(\R)}$ is \chg{the unique fixed point of $f_C$ different from the cusp}, and the indeterminate points of $\check f^{-1}$ are given by
$
p_i^- = \tilde\gamma(t_i) \prec p_{fix},
$
where
\begin{itemize}
 \item $t_i = \frac1{1-\delta^{n_i}}$ if $\sigma(i) = i$;
 \item $t_i = \frac{1+\delta^{n_j}}{1-\delta^{n_i+n_j}}$ if $\sigma$ exchanges $i$ and $j$;
 \item $t_i = \frac{1+\delta^{n_k}+\delta^{n_j+n_k}}{1-\delta^{n_i+n_j+n_k}}$ if $\sigma:i\mapsto j\mapsto k$ is cyclic;
\end{itemize}

Conversely, given orbit data $n_1,n_2,n_3,\sigma$ and a real root $\delta \neq \pm 1$ of the associated characteristic polynomial, suppose that $t_i$ are given by the above formula.  If the parameters $\{\delta^jt_i:1\leq i \leq 3, 0\leq j \leq n_i-1 \}$ are all distinct, then there is a basic real map, unique up to linear conjugacy, that \cg{properly fixes $C$} and realizes the orbit data $n_1,n_2,n_3,\sigma$.
\end{thm}

We note regarding the second half of this theorem that it is difficult to tell in general when the parameters of interest are distinct.  This issue is discussed at length in \cite{mcm} and \cite{dil}. \qu{In fact, even when the parameters are not distinct, one can find a (possibly degenerate) quadratic birational map that properly fixes $C$ with the correct determinant $\delta$ and whose indeterminate points have forward $f_C$ orbits described by the parameter values $t_i$ in Theorem \ref{thm:realizable}.  For instance, if $n_1 = n_2 = n_3 = 4$, one obtains a quadratic birational map $\check f$ for which all three indeterminate points coincide and map under $\check f^4$ to the lone (triple) indeterminate point for $\check f^{-1}$. The map is in fact independent of the permutation $\sigma$.  However, in a strict sense elaborated more fully in the discussion around Theorems 3.5 and 3.6 in \cite{dil}, this map correctly realizes the orbit data only for the case $\sigma = \id$.}

\section{Homology growth for real automorphisms}
\label{sec:action}

Let $\check f:\cp^2\to\cp^2$ be a basic real map that properly fixes the cuspidal cubic curve $C$ and realizes the orbit data $\sigma,n_1,n_2,n_3$ with determinant $\delta>1$.  Then by restriction we have an associated diffeomorphism $f_\R:X(\R)\to X(\R)$ of the real slice of $X$.  Our goal in this section is to give an effective means to work out the pushforward action $(f_\R)_*:H_1(X(\R);\R)\to H_1(X(\R);\R)$.

We begin by making an observation that, while not needed in the remainder of the section, is interesting in its own right and will be useful to us in Section \ref{sec:reallyreal} below.  Recall that a polynomial $P(t)$ is reciprocal if $P(t) = t^{\deg P} P(1/t)$.  If $P(t)$ is the characteristic polynomial of an invertible linear operator $T$, then $t^{\deg P}P(1/t)$ is the characteristic polynomial of $T^{-1}$.  Hence $P$ is reciprocal if $T^{-1}$ is conjugate to $T$.

\begin{thm}
\label{thm:reciprocal}
The characteristic polynomial of $(f_\R)_*$ is reciprocal, equal to that of $(f_\R^{-1})_*$.
\end{thm}

As is well-known, intersection theory considerations imply that the characteristic polynomial is reciprocal for the action $f_*:H_2(X;\R)\to H_2(X;\R)$ associated to the ambient complex automorphism.  The problem with $(f_\R)_*$ is that the real surface $X(\R)$ is non-orientable, so that one cannot employ intersection theory directly.

\begin{proof}
Since $C$ is $f$-invariant, the real map $f_\R$ restricts to a diffeomorphism on the open set $X(\R)\setminus C$.  We claim that the inclusion map 
$\iota:X(\R)\setminus C\hookrightarrow X(\R)$ induces an isomorphism 
$\iota_*:H_1(X(\R)\setminus C;\R)\to H_1(X(\R);\R)$. To see this, recall that in affine coordinates $(x,y)$, the regular part of $C$ is the set $\{y=x^3\}$.  Note that the image of $\iota_*$ includes the class of the proper transform $\ell_j$ of each real horizontal line $\{y=x_j^3\} \cap \cp^2(\R)$ that passes through a point $(x_j,x_j^3)$ blown up to obtain $X$.  By
Proposition \ref{prop:classofrealline}, the set $\{\ell_1,\dots,\ell_N\}$ of such lines is independent in and therefore generates $H_1(X(\R);\R)$.  On the other hand, the open set $X(\R)\setminus C$ can be retracted onto the union of precisely these lines, all of which meet at a single point in $X(\R)\setminus C$.  Hence their classes generate $H_1(X(\R)\setminus C;\R)$, too, which proves the claim.

It now suffices to show the characteristic polynomial is reciprocal for the action of $(f_\R)_*$ on $H_1(X(\R)\setminus C;\R)$.  The advantage is that $X(\R)\setminus C$ is orientable, with volume form  given by the restriction of the real meromorphic two form $\omega$ on $X$ with a simple pole along $C$.  Hence there is a well-defined and non-degenerate intersection form $\pair{\cdot}{\cdot}$ on $H_1(X(\R)\setminus C;\R)$.  Since $f^*\omega = \delta \omega$, the diffeomorphism $f_\R$ preserves orientation on $X(\R)\setminus C$ and therefore also the intersection form.  \chg{That is $(f_\R^{-1})_* = (f_\R)_*^{-1}$ is the intersection adjoint of $(f_\R)_*$ on $H_1(X(\R)\setminus C;\R)$.  Linear operators are conjugate to their adjoints, so $(f_\R)_*$ and $(f_\R^{-1})_*$ have the same characteristic polynomials. }
\end{proof}

Continuing to use the notation from the previous section, we let $e_{i,j} := E_{i,j}\cap X(\R)$ denote the real slice of each exceptional curve for the blowup $X\to \cp^2$.  We give $e_{i,j}$ the `clockwise' orientation described \cg{in Section \ref{sec:homology}}.  Similarly, we let $e_i^\pm := E_i^\pm \cap X(\R)$ denote the (strict transforms of the) real lines obtained by intersecting the critical lines of $\check f^{\pm 1}$ with $X(\R)$.  Each of these meets $C$ at two of the points  $p_j^\pm$.  Hence each intersection $e_i^\pm\cap C$ contains three distinct points, all in $\cg{\mathbb{A}^2(\R)}$.  In particular, $e_i^\pm$ is neither the line at infinity nor any vertical line.  By our convention above, all $e_i^\pm$ are oriented from left to right in $\cg{\mathbb{A}^2(\R)}$.

The hypotheses that $\check f$ properly fixes $C$ with positive determinant implies that the restriction $f_C$ preserves orientation along $C_{\R}$.  Hence at any non-indeterminate point $p\in C_{\R}$, we have that $\check f$ preserves \chg{the two-dimensional orientation} on $\cg{\mathbb{A}^2(\R)}$ if and only if locally near $p$, $\check f$ preserves the two components of $\cg{\mathbb{A}^2(\R)}\setminus C_{\R}$.  When this happens, we say that $\check f$ is \emph{orientation-preserving at $p$}, even though $\cg{\cp^2(\R)}$ is non-orientable.

\begin{prop}
\label{prop:orientation}
For each \qu{$p\in C(\R)$} \chg{that is not indeterminate or critical for $f$}, we have that $\check f$ is orientation preserving at $p$ if 
$$
\# \{i:f_C^{-1}(p_i^-)\prec p\} + \#\{i:p_i^+ \prec p\}
$$
is even and orientation reversing otherwise.
\end{prop}

\noindent\qu{Keep in mind here that there are only six indeterminate points $p_i^+, p_i^-$ in total for $f$ and $f^{-1}$; also that \qu{$f_C^{-1}(p_i^-) := (\check{f}^{-1}|_C)(p_i^-) \in C(\R)$} is a point for each $i\in\{1,2,3\}$, even though $p_i^-$ is indeterminate for $\check f^{-1}$.}

\begin{proof}
\qu{
We claim first that $\check f$ is orienation preserving at all \qu{$p\in C(\R)$} near the cusp, i.e. near the line at infinity.  This can be seen by employing new affine coordinates $(x',y') = (1/y,x/y)$ identifying the cusp with $(0,0)$ and $C$ with $(y')^3 = (x')^2$.  The fact that near $(x',y') = (0,0)$, the map $\check f$ is a local diffeomophism preserving $\{(y')^3 = (x')^2\}$ means that the differential $D_{(0,0)}\check f$ is diagonal of the form
$$
\left[\begin{matrix} \alpha^3 & 0 \\ 0 & \alpha^2\end{matrix}\right].
$$
Indeed $1/\alpha = \delta$ is just the determinant of $\check f$.  It follows for all $(x',y')$ near $(0,0)$ that $D_{(x',y')}f$ must preserve the two components of the complement of $\{(y')^3 = (x')^2\}$.  Conjugating back to our old affine coordinate, we find that $\check f$ preserves orientation about any point \qu{$p\in C(\R)$} close enough to infinity.}


Moreover, $\check f$ cannot change from orientation preserving to reversing or vice versa except at points \qu{$p\in C(\R)$} where $D_p \check f$ is singular or undefined, so it remains to understand what happens near points where $\check f$ is critical or indeterminate. 

\qu{Let us consider first the case when $p$ moves past a non-indeterminate critical point \qu{$q\in C(\R)$} for $\check f$.  That is, $q = f_C^{-1}(p_i^-)$ is the unique non-indeterminate point in $e_i^+\cap C$ for some $i\in\{1,2,3\}$.  We can choose local coordinates for source and target that identify $q$ and $p_i^-$ with $(0,0)$, \qu{$C(\R)$} with the horizontal axis oriented from left to right, and $e_i^+$ with the vertical axis so that points above \qu{$C(\R)$} in affine coordinates are above the horizontal axis in local coordinates.  Since $\delta >0$, it follows that $(1,0)$ is an eigenvector for $D_{(x,0)}\check f$ with positive eigenvalue at all points $(x,0)$ on the horizontal axis.  On the other hand $\det D_{(x,0)} \check f$ has a simple zero along the vertical axis, so the remaining eigenvalue of $D_{(x,0)}\check f$ changes sign as $x$ passes $0$ and the eigenvector for this eigenvalue is transverse to the horizontal axis.  It follows that the effect of $\check f$ on orientation switches as \qu{$p\in C(\R)$} passes $f_C^{-1}(p_i^-)$; i.e. if the differential preserves the upper and lower half planes for $x<0$, then $D_{(x,0)}\check f$ reverses the half planes when $x>0$, and vice versa.  
}



Similarly, the effect of $\check f^{-1}$ on orientation switches whenever $p$ moves past a point $f_C(p_i^+)$.  But this is the same as saying that $\check f$ changes effect on orientation when $p$ moves past an indeterminate point $p_i^+$. 

All told, by decreasing the parameter of $p=\gamma(x)$ to $-\infty$, we see that $\check f$ preserves orientation at $p$ precisely when there are an even number of indeterminate or critical points $p_i^+, f_C^{-1}(p_i^-)$ in between.
\end{proof}

In order to understand the induced action $\cg{(f_\R)_*}$ on $H_1(X(\R);\R)$ it suffices to know what happens to the real exceptional curves $e_{i,j}$.  The following result, together with Proposition \ref{prop:classofrealline} gives a practical means for extracting this information from given orbit data.

\begin{thm}
\label{thm:realaction}
The action $\cg{(f_\R)_*}:H_1(X(\R))\to H_1(X(\R))$ is given on generators $e_{i,j}$ by
\begin{itemize}
\item $\cg{(f_\R)_*} e_i^+ = \pm e_{i,1}$, where the sign is positive if and only if the unique non-indeterminate point $f_C^{-1}(p_i^-)$ in $e_i^+\cap \qu{C(\R)}$ is preceded (in $\qu{C(\R)}$) by an odd number of the three indeterminate and critical points for $\check f$ in $\qu{C(\R)}\setminus e_i^+$.
\item if $1\leq j \leq n_i-1$, then $\cg{(f_\R)_*} e_{i,j} = \pm e_{i,j+1}$, where the sign is positive if and only if $p_{i,j}$ is preceded by an even number of indeterminate and critical points $p_k^+,f_C^{-1}(p_k^-)\in \qu{C(\R)}$ for $\check f$;
\item $\cg{(f_\R)_*} e_{i,n_i} = \pm e_{\sigma(i)}^-$, where the sign is positive if and only if $f_C(p_{\sigma(i)}^+) \in \qu{C(\R)}$ is preceded by an odd number of the three indeterminate and critical points for $\check f^{-1}$ in $\qu{C(\R)}\setminus e_{\sigma(i)}^-$.
\end{itemize}
\end{thm}

\begin{proof}
We deal with the second item first.
Since $p_{i,j+1} = \check f(p_{i,j})$ for $1\leq j\leq n_i-1$, we have $f_* e_{i,j} = \pm e_{i,j}$.  As $e_{i,j}$ identifies with the projectivization of the tangent space at $p_{i,j}$, the sign is determined by whether or not $\check f$ preserves orientation at $p_{i,j}$, i.e. by the criterion in Proposition \ref{prop:orientation}.  Hence the criterion for the sign of $\cg{(f_\R)_*} e_{i,j}$ in this theorem follows directly from that one.  

The arguments for the first and third items are similar so we deal only with the first. The real automorphism $f_\R:X(\R)\to X(\R)$ maps $e_i^-$ diffeomorphically onto $e_{i,1}$.  So $\cg{(f_\R)_*}e_i^+ = \pm e_{i,1}$, and the sign will be determined by the image of the forward tangent vector $v$ to $e_i^+$ at the unique non-indeterminate point $f_C^{-1}(p_i^-)$ of $e_i^+\cap \qu{C(\R)}$.  

Moreover, to understand the image of $v$ it is perhaps easier to consider the image of a parallel translate $\tilde v$ originating from a point $p\in \qu{C(\R)}$ slightly preceding $f_C^{-1}(p_i^-)$.  If $D_p \check f_\R(\tilde v)$ points above $\qu{C(\R)}$ at $\check f(p)$, then the $f_\R$ maps $e_i^+$ about $e_{1,1}$ in a clockwise fashion, i.e. $\cg{(f_\R)_*}e_i^+ = e_{1,1}$.  Otherwise $\cg{(f_\R)_*}e_i^+=-e_{1,1}$.

Hence the sign of $\cg{(f_\R)_*} e_i^+$ will be positive if and only if we are in one of two cases:
\begin{itemize}
 \item $\tilde v$ itself points above $\qu{C(\R)}$ and $\check f_\R$ is orientation preserving at $p$;
 \item or $\tilde v$ points below $\qu{C(\R)}$ and $\check f_\R$ is orientation reversing at $p$.  
\end{itemize}
Since any non-vertical line passes above all points $(x,y)\in \qu{C(\R)}$ with $x <<0$, we have that $\tilde v$ points above $\qu{C(\R)}$ if and only if $f_C^{-1}(p_i^-)$ lies between the other two points $p_j^+,p_k^+$, $j\neq k\neq i$ in the intersection $e_i^+\cap \qu{C(\R)}$.  So one infers the first item in the present theorem from these observations and Proposition \ref{prop:orientation}.
\end{proof}

We now prove Theorem \maximal \ by applying Theorem \ref{thm:realaction} to the relevant orbit data.  Let $\sigma$ be the cyclic permutation $(123)$. \cg{Using the fact that parameters for the indeterminate points of $\check f^{-1}$ is given by rational functions of a root $t$ of the characteristic polynomial and their sum is equal to $t-2$, we see} 
the characteristic polynomial associated to any orbit data of the form $n_1,n_2,n_3,\sigma$ is given by (see Theorem A.1 in \cite{beki3} and Equation (2) in \cite{dil})
\begin{equation}
\label{eqn:cycliccharpoly}
\chi(t) = t-t^{n_1+n_2+n_3}+(t-1)(t^{n_1}+1)(t^{n_2}+1)(t^{n_3}+1).
\end{equation}

\begin{prop}
\label{prop:charpoly}
The characteristic polynomial $\chi(t)$ has a real root $t=\delta >1$ if and only if $n_1+n_2+n_3 \geq 10$.  In fact, this root is increasing in each of the orbit lengths $n_j$.
\end{prop}

\begin{proof}
One easily computes from \eqref{eqn:cycliccharpoly} that $\chi(1) = 0$, $\chi(2) > 1$ and $\chi'(1) = 9-n_1-n_2-n_3$.  So when $n_1+n_2+n_3 \geq 10$, it follows from the intermediate value theorem that $\chi$ has a real root between $1$ and $2$.  On the other hand, it follows from general geometric considerations (see e.g. \cite[Proposition 2.2]{dil}) that all roots of $\chi(t)$ have magnitude $1$ whenever $n_1+n_2+n_3\leq 9$.

Since $\chi$ has at most \emph{one} root $\delta>1$ (see  \cg{\cite[Theorem~0.3]{difa}}), it follows (when $\delta$ exists) for $t>1$ that $\chi(t) > 0$ if and only if $t>\delta$.  Thus it suffices to show that if $\delta>1$ is a root of $\chi(t)$, then $\tilde\chi(\delta)<0$, where $\tilde\chi$ is the polynomial obtained from \eqref{eqn:cycliccharpoly} by replacing $n_1$ with $n_1+1$.  In fact, 
\begin{equation*}
\begin{aligned}
\tilde\chi(\delta) &= \tilde\chi(\delta)-\delta\chi(\delta) \\
&\cg{= \delta (1-\delta) - (1-\delta)(1+\delta^{n_2})(1+\delta^{n_3}) [ (1+\delta^{n_1+1}) - \delta (1+\delta^{n_1})]}\\
&= \delta (1-\delta) - (1-\delta)^2(1+\delta^{n_2})(1+\delta^{n_3}) < 0.
\end{aligned}
\end{equation*}
\end{proof}

\subsection{The Coxeter case.} In this subsection we deal with the first case of Theorem \maximal.  That is, we fix our orbit data to be $n_1=n_2=1$, $n_3=n\geq 8$ and $\sigma = (123)$ cyclic.  This situation is of particular interest (see \cite{mcm}) because it corresponds to the so-called Coxeter element in a certain infinite reflection group.  Regardless, the characteristic polynomial \eqref{eqn:cycliccharpoly} specializes to
\begin{equation}
\label{eqn:poly11n}
0=\chi_n(t) := t^n (t^3-t-1) + t^3+t^2-1
\end{equation}
Take $\delta = \delta_n > 1$ to be the largest real root of $\chi_n$.

To obtain a realization $\check f$ of our orbit data and then understand the action  $\cg{(f_\R)_*}$ of the associated automorphism $f_\R:X(\R)\to X(\R)$, we need to understand how the critical orbits of $\check f$ must be distributed along $C$.  In this case, these are $p_{1,1} = p_1^- = p_2^+$, $p_{2,1} = p_2^- = p_3^+$, and $p_3^- = p_{3,1} \mapsto \dots \mapsto p_{3,n} = p_1^+$.  

\begin{lem}
\label{lem:order} When $n\geq 8$, there exists a basic map $\check f$ realizing the orbit data $1,1,n,(123)$ and properly fixing the cuspidal cubic $C$ with determinant $\delta>1$.
If $p_{fix}$ is the unique (finite) fixed point of $f_C$, then we may suppose that $p_i^\pm \prec p_{fix}$ for all $i=1,2,3$.  Moreover, 
$$
f_C^2(p_2^+) \prec p_1^+ \prec f_C(p_2^+) = f_C^3(p_3^+).
$$
\end{lem}

\begin{proof} Suppose for the moment that $\check f$ is the realization we seek.
Theorem \ref{thm:realizable} tells us that the parametrization $\tilde \gamma$ of $C$ can be adjusted so that $p_{fix} = \tilde\gamma(0)$ and the indeterminate points of $\check f^{-1}$ are given by $p_i^- = \tilde\gamma(t_i)$, where
\begin{equation*}
\begin{aligned}
t_1 \ &=\ Q^{-1}(1+\delta + \delta^{n+1}) \\
t_2 \ &=\ Q^{-1}(1+\delta^n+ \delta^{n+1}) \\
t_3 \ &=\ Q^{-1}(1+\delta + \delta^2), \\
\end{aligned}
\end{equation*}
for $Q= 1 - \delta^{n+2} <0$.  In particular $t_i<0$ for each $i$.  Since $p_{fix} = \tilde\gamma(0)$, we see that $p_i^- \prec p_{fix}$.   Since $p_{\sigma(i)}^+ = \tilde\gamma(\delta^{n_i-1}t_i)$, we see that $p_i^+ \prec p_{fix}$ for each $i$, too.  

The formulas for the parameters $t_i$ also yield 
\begin{equation*}
\begin{aligned}
& Q(\delta^2 t_3^+ - t_2^+)\ =\  \delta^n (\delta^3-\delta-1) + \delta^3+ \delta^2-1 =0,\\
&Q(t_1^+ - \delta t_2^+) \,\ =\  - \delta^n (\delta^3-\delta-1) - \delta ^2 =  \delta^3-1 > 0 \\
&Q(\delta^2 t_2^+ - t_1^+) \ =\ \delta^n (\delta^3-\delta-1) + \delta^{n-1} (\delta^3-1) + \delta^2 = (\delta^3-1) ( \delta^{n-1} - 1) >0  
\end{aligned}
\end{equation*}
where $t_{\sigma(i)}^+ = \delta^{n_i-1}t_i$. That is, $f_C^2(p_3^+) = p_2^+$ and $f_C^2(p_2^+) \prec p_1^+ \prec f_C(p_2^+)$, as asserted.

Finally, note that these computations show that the parameters $t_1$, $t_2$ and $\delta^j t_3$ are distinct for all $j\in\N$.  Hence the last part of Theorem \ref{thm:realizable} ensures the existence of the realization $\check f$ we have so far taken for granted.
\end{proof}

Lemma \ref{lem:order} and the fact that $f_C$ preserves order $\prec$ along $C$ implies that the (extended) critical orbits of $f$ are ordered as shown in Figure \ref{fig:11norder}.  From the figure, Theorem \ref{thm:realaction}, and Proposition \ref{prop:classofrealline}, we can easily deduce the action $\cg{(f_\R)_*}$.

\begin{figure}
\centering
\def\svgwidth{6.5truein}
  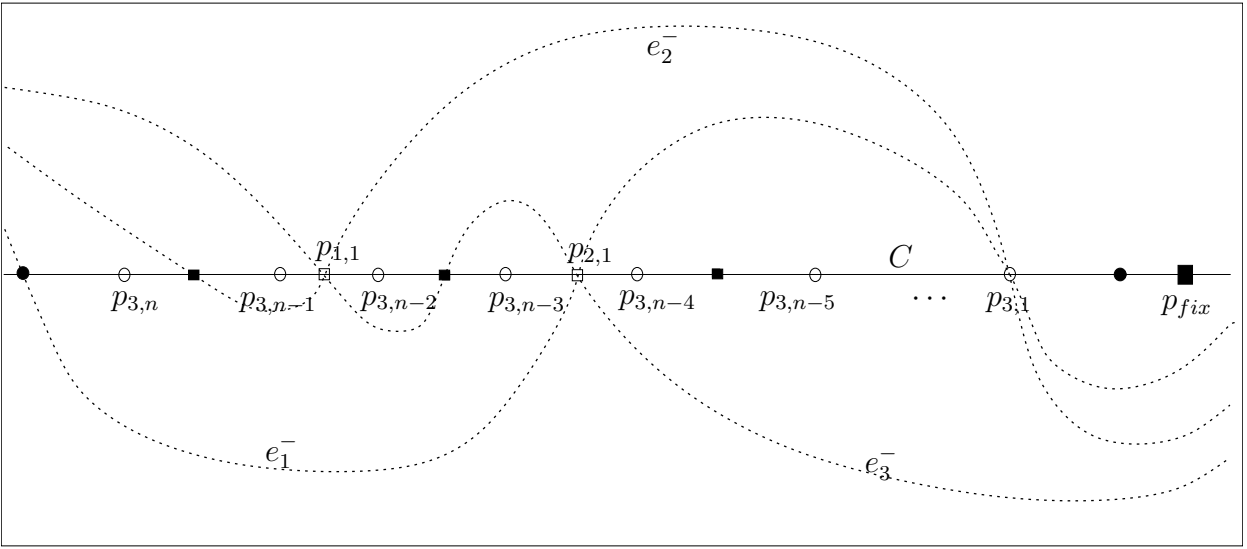
\caption{Critical orbits for the basic map realizing orbit data $n_1=n_2 = 1,n_3 = n, \sigma =(123)$.  The center horizontal line represents the cuspidal cubic $C$.  Dashed curves are the real lines $e_i^- = \check f(p_i^+)$, $i=1,2,3$.  Hollow circles/squares are points blown up to obtain the surface $X$.  Unlabeled solid circles and rectangles represent the images $f_C(p_i^+)$ and $f_C^{-1}(p_i^-)$ along $C$ of forward and backward indeterminate points. \label{fig:11norder}}
\end{figure}

\begin{cor}
\label{cor:coxeteraction}
The action $f_{\R *}:H^1(X;\R)\to H^1(X;\R)$ is given by
\begin{eqnarray*}
e_{1,1} & \mapsto & e_2^- \sim e_{2,1}+e_{3,2}-\dots +e_{3,n-2} - e_{3,n-1} + e_{3,n},\\ 
e_{2,1} & \mapsto & -e_3^- \sim e_{3,1}+\dots + e_{3,n-4} -e_{3,n-3}+e_{3,n-2}- e_{3,n-1}-e_{3,n},\\
e_{3,n} & \mapsto & -e_1^- \sim -e_{3,2} - \dots - e_{3,n-4} + e_{1,1}+ e_{3,n-3}+ \dots +e_{3,n}.\\
e_{3,j} & \mapsto & 
\left\{\begin{array}{rcl} e_{3,j+1} & \text{ if } j=n-2 \text{ or } j=n-4 \\
                          -e_{3,j+1} & \text{ for all other } 1\leq j \leq n_3-1
       \end{array}\right.
\end{eqnarray*}
\end{cor}

\begin{proof}
Since, for instance, $n_1 = 1$ and $\sigma(1)=2$, we have that $p_{1,1} = p_2^+$, and therefore $\cg{(f_\R)_*} e_{1,1} = \pm e_2^-$.
The third item in Theorem \ref{thm:realaction} tells us that the sign is determined by how many of the three points $f_C(p_1^+), f_C(p_3^+), p_2^-\in \qu{C(\R)} \setminus e_2^-$ precede $f_C(p_2^+) = f_C(p_{1,1})$.  In Figure \ref{fig:11norder}, these are the intersections of $e_1^-,e_3^-$ with $\qu{C(\R)}$ away from $e_3^-$.  Precisely one of them $f_C(p_1^+)$ precedes $f_C(p_{1,1})$ (the leftmost point in $e_2^-\cap \qu{C(\R)}$), so Theorem \ref{thm:realaction} tells us $\cg{(f_\R)_*} e_{1,1} = e_2^-$.  From Figure \ref{fig:11norder} again and Proposition \ref{prop:classofrealline} it is further evident that $e_2^-  \sim e_{2,1}+e_{3,2}+\dots +e_{3,n-2} - e_{3,n-1} + e_{3,n}$.  This completes our computation of $\cg{(f_\R)_*} e_{1,1}$.
The images of $e_{2,1}$ and $e_{3,1}$ are computed in the same way.

To find $\cg{(f_\R)_*} e_{3,j}$ for $1\leq j\leq n-1$, we note from Figure \ref{fig:11norder} that there are an even number of critical/indeterminate points in $\qu{C(\R)}$ preceding $p_{3,j}$ unless $j=n-2$ or $j=n-4$.  So the second item in Theorem \ref{thm:realaction} implies the formula for $\cg{(f_\R)_*} e_{3,j}$ given in the present corollary.  
\end{proof}

%

From Corollary \ref{cor:coxeteraction} one can write down the matrix for \chg{$(f_\R)_*$ relative to the generators $e_{i,j}$ for $H_1(X(\R);\R)$ and compute the characteristic polynomial for $(f_\R)_*$ directly from that.  An important point, here and elsewhere, is that regardless of $n$, the matrix is triangular outside of three columns (those corresponding to $e_{1,1}$, $e_{2,1}$, and $e_{3,n}$).  Hence the same few row operations suffice in all cases to put the matrix for $f_* - t \,\id$ in diagonal form outside these three columns. Comparing with \eqref{eqn:poly11n}, one then verifies that the characteristic polynomial of $\cg{(f_\R)_*}:H_1(X(\R);\R)\to H_1(X(\R);\R)$ is equal to $(t-1)^{-1}\chi_n(t)$.}  In particular, the spectral radius $\rho(\cg{(f_\R)_*})$ \chg{is the} same as that of the action $f_*:H_2(X;\R)\to H_2(X;\R)$.  It follows from the entropy bounds of Yomdin and Gromov that $h_{top}(f_\R) = h_{top}(f) = \log\rho(f_*)$, i.e. Theorem {\maximal } holds for Coxeter orbit data.

\subsection{The second case of Theorem {\maximal}} The orbit data $\sigma=(1\,2\,3)\ n_1=2, n_2= n_3=n\ge 4$ can be dealt with in the same fashion, so we only summarize.  The characteristic polynomial for this orbit data is given by
$$
(t-1) (t^2+1) (t^n+1)^2 - t^{2 n+2}+1
$$
Again taking $\delta>1$ to be the largest root, and letting $p_{fix} \in \qu{C(\R)}$ be the unique finite fixed point of $f_C$, one deduces the following analogue of Lemma \ref{lem:order}:

\begin{lem} There exists a basic map $\check f$ that properly fixes $C$ and realizes the orbit data $2,n,n,(123)$ with determinant $\delta$.  If $p_{fix}$ is the unique (finite) fixed point of $f_C$, then $p_i^\pm \prec p_{fix}$ for all $i=1,2,3$.  Moreover, 
$$
f_C(p_2^+) \prec p_1^+ \prec p_3^+ \prec p_2^+.
$$
\end{lem}

This Lemma allows one to work out the action of $\cg{(f_\R)_*}$ on generators of $H_1(X(\R);\R)$, and a little further computation then shows as before that the characteristic polynomials of $\cg{(f_\R)_*}$ and of $f_*:H_2(X;\R)\to H_2(X;\R)$ differ by a factor of \chg{$t-1$}.
\qed

In the appendix to this paper, we give the analogues of Lemma \ref{lem:order} for \emph{all} possible orbit data with $\sigma = (123)$ cyclic.  This leads to a general expression for the characteristic polynomial for $\cg{(f_\R)_*}$ in the cyclic case.  We also comment on the remaining possibilities for the permutation: $\sigma = \id$ and $\sigma = (12)$ a transposition.

\section{Real maps with non-maximal entropy}
\label{sec:saddle}

 In this section, we prove Theorem \nonmaximal, relying on the following fact.  It was proven by Bedford, Lyubich and Smillie \cite{belysm} for polynomial automorphisms of $\C^2$, but later foundational work of Cantat \cite{can1}, DeThelin \cite{det} and Dujardin \cite{duj} allows one to easily adapt the proof to real automorphisms of compact complex surfaces.  

\begin{thm}[{\cite[Corollary~8.3]{can2}}]
\label{thm:allrealsaddles}
Let $f:X\to X$ be an automorphism on a blowup $X\to\cp^2$ of the complex projective plane.  If $f$ is real, then $h_{top}(f_\R) = h_{top}(f)$ if and only if all saddle periodic points of $f$ are contained in $X(\R)$ or in $f$-invariant algebraic curves. 
\end{thm}

The theorem proceeds in turn from the existence and uniqueness of a measure of maximal entropy for $f$ on $X$, and the additional property that support of the measure contains all the saddle periodic points of $f$ \chg{outside periodic curves}.  The survey \cite{can2} gives a good detailed account.  

Theorem \nonmaximal \ is an immediate consequence of Theorem \ref{thm:allrealsaddles} and

\begin{thm}
\label{thm:33n}
For every $n\geq 4$ and $\sigma = (123)$ cyclic, there exists a basic real map $\check f_n$, \chg{unique up to linear conjugacy}, that properly fixes $C$ with determinant $\delta > 1$ and realizes the orbit data $3,3,n,\sigma$.  The associated automorphism $f_n:X_n\to X_n$ has two complex fixed points in $X_n\setminus X_n(\R)$ but no invariant algebraic curves other than $C$.  When $n$ is large enough the complex fixed points are saddles.
\end{thm}

\begin{figure}
\centering
\def\svgwidth{6.5truein}
  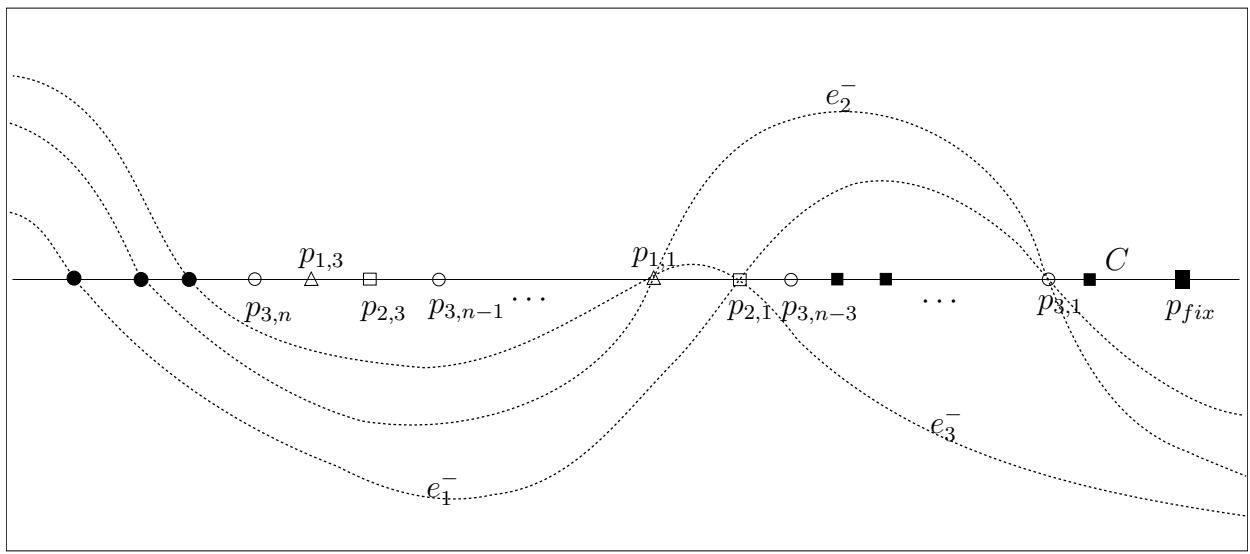
\caption{Critical orbits for the basic map realizing orbit data $n_1=n_2 = 3,n_3 = n, \sigma =(123)$.  The center horizontal line represents the cuspidal cubic $C$.  Dashed curves are the real lines $e_i^- = \check f(p_i^+)$, $i=1,2,3$.  Hollow circles/triangles/squares are points blown up to obtain the surface $X$.  Unlabeled solid circles and rectangles represent the the images $f_C(p_i^+)$ and $f_C^{-1}(p_i^-)$ along $C$ of forward and backward indeterminate points. \label{fig:33norder}}
\end{figure}

For the orbit data in the Theorem, the characteristic polynomial \eqref{eqn:cycliccharpoly} specializes to 
\begin{equation}
\label{E:cplxchar}
\chi_n(t) = h(t) - t^n \cdot t^7 h(1/t),
\quad{\text{where}}\quad
h(t) = t^7-t^6+2 t^4-2 t^3 + 2 t -1.
\end{equation}
Proposition \ref{prop:charpoly} tells us that when $n\ge 4$ this polynomial has a (necessarily unique) real root $\delta$ larger than $1$ and that this root increases with $n$.  From the formula for $\chi_n$ one sees that the limiting value of $\delta$ as $n\to \infty$ is the largest real root of $t^7h(1/t)$.  The monotonicity assertion in Proposition \ref{prop:charpoly} therefore implies that the $\delta$ lies between the largest roots of $\chi_4$ and of $t^7 h(1/t)$, i.e. by finding these roots numerically, $\delta \in (1.431,1.684)$ for all $n\in\N$. 

Let us note further that $t=1$ is always a simple root of $\chi_n$, since $\chi_n(1) = h(1) -h(1) = 0$ and $\chi_n'(1) = 2h'(1) - (n+7)h(1) = 3-n \leq -1$.

The existence of the basic map $\check f_n$ in Theorem \ref{thm:33n} is proved along the same lines as it was for Coxeter orbit data in Lemma \ref{lem:order}.  This time the ordering of the critical orbits is determined by the inequalities
$$
f_{n,C}(p_3^+) \prec p_1^+ \prec p_2^+ \prec p_3^+ \prec p_{fix}
$$
and is displayed in Figure \ref{fig:33norder}. From now on $f_n:X_n\to X_n$ will denote the complex surface automorphism obtained from $\check f$ by blowing up the critical orbits of $f_n$.

\begin{lem}
The only $f_n$ invariant algebraic curve is $C$.  Counting multiplicity, there are four fixed points for $f_n$, exactly two of which do not lie in $C$.
\end{lem}

\begin{proof}
The homology class of any $f_n$ invariant curve is also invariant, i.e. it is an eigenvector of $f_{n*}:H_2(X_n;\R)\to H_2(X_n;\R)$ with eigenvalue $1$. Let $V$ be an $f_n$ invariant algebraic curve.  Since $C$ is invariant and $1$ is a simple root of $\chi_n$, it follows that the homology class of $V$ is a multiple of the homology class of $C$.  However, the self-intersection of $C$ is $3^2 - N$, where $N$ is the number of points in $C$ that are blown up.  Since $N\geq 10$ under the hypotheses of Theorem \ref{thm:33n}, $C^2 <0$.  It follows that $V$ itself (i.e. as a divisor) must be a multiple of $C$.

To count fixed points, we appeal to the Lefschetz formula \chg{\cite[Chapter 3.4]{grha} which tells us for automorphisms of complex rational surfaces that the number of fixed points of $f$, counted with multiplicity,} is two more than the trace of \cg{$f_{n*}:H_2(X_n;\R)\to H_2(X_n;\R)$.}  This trace is $2$, i.e. the coefficient of $t$ in the polynomial $h$ in the formula for $\chi_n$.
Since \cg{$f_{n,C}$} has two fixed points \chg{and fixed points for holomorphic maps always have positive multiplicity}, this leaves two fixed points in the complement of $C$.
\end{proof}

Restricting $f_n$ to a diffeomorphism $f_{n,\R}$ on the real slice $\qu{X_n(\R)}$ of $X_n$, one computes the action on $f_{n,\R*}:H_1(X_{n}(\R);\R)\to H_1(X_{n}(\R);\R)$ as in Corollary \ref{cor:coxeteraction}.  Following the method of the previous section, one finds that the characteristic polynomial $ \phi_n(t)$ of $f_{n,\R*}$ is given by
\begin{equation}\label{E:realchar}(t+1) \phi_n(t) = g(t) + (-t)^n t^7 g(1/t)  
\qquad\text{where}\qquad
g(t) = 1- 2t^4+2t^5-t^6+t^7.
\end{equation}

Though it isn't necessary for the proof of Theorem \ref{thm:33n}, we include the following by way of contrast to Theorem \maximal.  

\begin{prop}
For $n\ge 4$, the spectral radius of the real homology action of $f_{n,\R}$ is strictly smaller than the dynamical degree of $f_n$. 
\end{prop}

\begin{proof}
The proposition can be checked by direct numerical computation for $n=4$.  When $n=5$, one also verifies numerically that $\delta > 1.5$.  From Proposition \ref{prop:charpoly}, we infer that $\delta > 1.5$ for all $n>4$.  To finish the proof, we will show that all roots of $\phi_n$ lie inside the disk $B := \{z\in\C:|z|<1.5\}$

Again, for $n=5,6$ this can be verified numerically.  When $n\ge 7$, we resort to Rouche's Theorem.  For $t \in \partial B\cg{ =  \{z\in\C:|z|=1.5\}}$, we have  
\[ |g(t)| \le 1+ 2|t|^4+ 2|t|^5 + |t|^6+|t|^7 < 55\] and
\begin{equation*}
\begin{aligned}
 |(-t)^n t^7 g(1/t)| \ &= \ |t|^n |1-t+2 t^2 - 2 t^3+t^7|\\ 
 & > |t|^n ( |t|^7 -1-|t|-2 |t|^2-2 |t|^3) >  3.3 \cdot (1.5)^n.
\end{aligned}
\end{equation*}
Notice that $3.3 \times (1.5)^7 >56$. It follows that if $n\ge 7$, for $t \in \partial B$ we have
\[ |(t+1) \phi_n(t)-(-t)^n t^7 g(1/t)| =|g(t)| <  |(t+1) \phi_n(t)|+|(-t)^n t^7 g(1/t)| \ \ \ \ \ \ t \in \partial B.\]
\cg{The polynomial $t^7g(1/t)$ has one real root $\sim -1.4334$ and  the moduli of the non-real complex roots of $t^7g(1/t)$ are $0.719, 0.980$ and $1.185$. }Thus every root of $(-t)^n t^7 g(1/t)$ lies in $B$.  Since $\deg (t+1)\phi_n(t) =  \deg (-t)^n t^7 g(1/t) = n+7$, it then follows from Rouch\'{e}'s theorem that every root of $\phi_n(t)$ has modulus smaller than $1.5$. 
\end{proof}

\begin{rem}
Adapting the argument in Section 2 of \cite{grhimc}, one can show for all $n\ge 4$ that $\phi_n$ is separable\cg{, i.e all roots are simple,} and that the only possible cyclotomic factors of $\phi_n$ are $(t-1), (t^2+1)$ and $(t^6-t^3+1)$. It follows that $\phi_n$ has a root outside the unit circle for all $n\ge 4$. Hence by Yomdin's bound, $h_{top}(f_{n, \mathbf{R}})$ is strictly positive.
\end{rem}

\subsection{Complex Saddle Fixed Points}

To complete the proof of Theorem \ref{thm:33n}, it remains to show \chg{that the two} fixed points of $f_{n,\R}$ that do not lie on $C$ are actually saddle fixed points outside $\qu{X_n(\R)}$.  For this we resort to an explicit formula for $\check f_n$.  The method for obtaining the formula is explained in \cite{bediki}.  Here we record only the result.  Namely, up to real linear conjugacy, $\check f_n = L\circ \sigma$, where $\sigma$ is, as before, the standard quadratic involution and $L=(\ell_{ij}) \in \text{GL}(3, \mathbb{C})$ is given by
\begin{equation*}
\begin{aligned}
& \ell_{11} \,=\, -\delta^4 (1-\delta+\delta^2) (1-\delta+\delta^3) (1-\delta^2+\delta^4)\\
& \ell_{12} \,=\, (1+\delta) (1-\delta+\delta^3)^2 (1-\delta+\delta^3-\delta^4+\delta^5)\\
& \ell_{13} \,=\, (1-\delta+\delta^2)(1-\delta^2+\delta^3) (1-\delta^2+\delta^4) (-1+\delta-\delta^3+\delta^4-\delta^5-\delta^6+\delta^7)\\
& \ell_{21} \,=\, -\delta^4 (1-\delta+\delta^3) (1-\delta^2+\delta^3) (1-\delta^2+\delta^3-\delta^5+\delta^6)\\
& \ell_{22} \,=\, (1+\delta^5) (1-\delta+\delta^3-\delta^4+\delta^5) (1-\delta+2 \delta^3-\delta^4-\delta^5+\delta^6)\\
& \ell_{23} \,=\, (1-\delta^2+\delta^4)(1-\delta+\delta^2-\delta^4+\delta^5) (-1+\delta-2 \delta^3+ 2 \delta^4-2 \delta^6+\delta^7).\\
& \ell_{31} \,=\, -\delta^7 (1-\delta^2+\delta^3)^2 \\
& \ell_{32} \,=\, \delta^3  (1-\delta+\delta^3) (1+\delta^5) \\
& \ell_{33} \,=\, \delta^3 (1-\delta^2+\delta^3) (1-\delta+\delta^2-\delta^4+\delta^5) (-1+\delta^2-\delta^3-\delta^4+\delta^5).\\
\end{aligned}
\end{equation*} 

To proceed, let us note that the line at infinity is critical for $\sigma$ and therefore also for $f_n = L\circ \sigma$.  In particular $f_n$ has no fixed point on the line at infinity.  So we are safe working in affine coordinates $(x,y) \mapsto [x,y,1]$, writing $f_n = (F_x,F_y)$.

\begin{lem} \label{L:cplxFix}
For $n\ge 4$, \cg{$f_n$} has two fixed points outside $\cg{\cp^2(\R)}$.
\end{lem}

\begin{proof} The point $(x,y)$ is fixed by $f_n$ if and only if $F_x(x,y)-x = 0 = F_y(x,y) - y$.  Using the above formulae for the coefficients of $L$, we find that the first equality reduces to $y=\xi(x)$ where  
$$
\xi(x) = - \frac{(\delta^5+1) x (N_0+N_1 x)}{D_0 + D_1 x+D_2 x^2}, 
$$
and the coefficients $N_i, D_i$ are the following integral polynomials in $\delta$:
\begin{equation*}
\begin{aligned}
&N_0 \ =\  - \delta^3(1-\delta+\delta^3)\\
& N_1 \ =\  (1-\delta+\delta^3-\delta^4+\delta^5)(1-\delta+2\delta^3-\delta^4-\delta^5+\delta^6)\\
& D_0 \ =\ \delta^7 (1-\delta^2+\delta^3)^2\\
&D_1\ = \ -\delta^3(\delta-1) (1-\delta^2+\delta^3) (1-\delta+3 \delta^3-2 \delta^4-\delta^5+4 \delta^6-2 \delta^7-\delta^8)\\
&D_2 \ =\  (1-\delta^2+\delta^4) (1-\delta+\delta^2-\delta^4+\delta^5)(-1+\delta-2 \delta^3+2 \delta^4-2 \delta^6+\delta^7). \\
\end{aligned}
\end{equation*}
So turning our attention to the second equality, we have that $(x,y)$ is a fixed point if and only if $y=\xi(x)$ and 
$$ 
0=F_y(x,y) - y \ = \ \frac{(x-1) S(x) Q(x) }{D(x)},
$$
where $D(x)$ is a polynomial with real coefficients, a linear function $S(x) = A_0 + A_1 x$ and a quadratic function $Q(x)= B_0 + B_1 x + B_2 x^2$ with
\begin{equation*}
\begin{aligned}
&B_0 \ =\  \delta^6(1-\delta+\delta^3) (1-\delta^2+\delta^3)\\
&B_1\ = \ -\delta^3(2-3\delta-2\delta^2+10 \delta^3-7\delta^4-7\delta^5+16 \delta^6-7\delta^7-7\delta^8+10\delta^9-2\delta^{10}-3\delta^{11}+2 \delta^{12})\\
&B_2 \ = \ (1-\delta^2+\delta^4)^2 (1-\delta+\delta^2-\delta^4+\delta^5) (1-\delta+\delta^3-\delta^4+\delta^5).
\end{aligned}
\end{equation*}
Hence
\begin{equation*}
\begin{aligned}
B_1^2 &- 4 B_0 B_2 \\= &\  -(\delta-1)^6 \delta^8 (\delta+1)^2 (1+\delta^3+\delta^6) (3-4 \delta - 4 \delta^2 + 11 \delta^3 - 4 \delta^4-4 \delta^5 + 3 \delta^6).
\end{aligned}
\end{equation*}
Notice that 
\begin{equation*}
\begin{aligned}
3-4 \delta - 4 \delta^2 + 11 \delta^3 &- 4 \delta^4-4 \delta^5 + 3 \delta^6 \\& =\  \delta^3 (5-\delta^2-1) + (\delta^3+1) [ 3 (\delta^2-1) (\delta-1) -1].
\end{aligned}
\end{equation*}
Since $1.4 < \delta<1.7$, we have
\[ \delta^2+1 <5, \quad 3 (\delta^2-1) (\delta-1) > 3\times (1.4^2-1) (1.4-1) = 1.152 \]
and therefore $B_1^2 - 4 B_0 B_2 <0$. Thus $Q(x)$ has two complex roots and $f_n$ has two complex fixed points.
\end{proof}

The proof of Theorem \ref{thm:33n} is now concluded by

\begin{lem}\label{lem:saddles}
For sufficiently large $n$, the two complex fixed points of $f_n$ are saddle. 
\end{lem}

\begin{center}
     \includegraphics[width=0.4\textwidth]{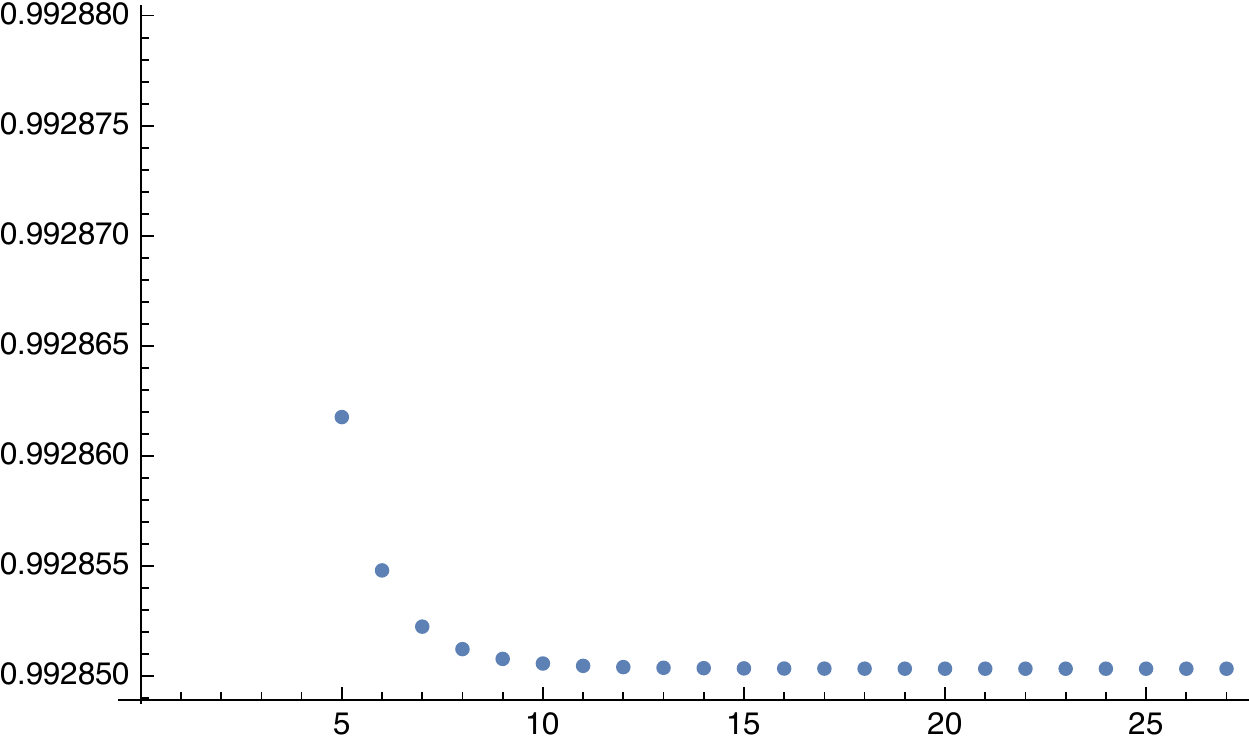}
  \captionof{figure}{The modulus of the smaller multiplier vs $n$. \label{figure:multi} }
\end{center}

\begin{rem} It seems plausible that Lemma \ref{lem:saddles}, and therefore also Theorem, \ref{thm:33n} work for \emph{all} $n\geq 4$.  When $n$ is small, one can confirm numerically that the complex fixed points of $f_n$ are saddles.  For example if $n=4$, the two multipliers have modulus $\approx 1.43903, 0.994417 $ and if $n=5$, the two multipliers have modulus $\approx 1.56666, 0.993212$.  In fact, we observe that the modulus of the smaller multiplier is decreasing in $n$ for $4\le n \le 30$ (see Figure \ref{figure:multi}) and seems to approach $0.99285\dots$ very quickly as $n$ grows. Since the product of the two multipliers is $\delta$, the modulus of the larger multiplier easily exceeds $1$. 
\end{rem}

\begin{rem} 
One might hope that the holomorphic version of the Lefschetz fixed point formula \cite[Chapter 3.4]{grha}   
$$
\sum_{f(p) = p}\frac{1}{\det(\id - Df(p))} = 1
$$
would be useful here and for proving Lemma \ref{lem:twocycles} below in a more conceptual fashion.  That is, the multipliers at the two fixed points on $C(\R)$ are known (see \cite[Lemma 5.2]{ueh2}), and the product of the multipliers at each of the other two fixed points must equal $\delta$.  This leaves us to determine only one multiplier at each of the latter.  The holomorphic fixed point formula gives us one additional relation between these mutlipliers which is not sufficient to identify them completely.
\end{rem}

\begin{proof}[Proof of Lemma \ref{lem:saddles}]
Since $f_n^*\omega = \delta^{-1}\omega$, where $\omega$ is the meromorphic two form with a simple pole along $C$, and since the complex fixed points of $f_n$ do not lie on $C$, we have that the multipliers $\mu_1,\mu_2$ of $Df_n$ at either of these points satisfy $\mu_1\mu_2 = \delta$.

After a (long) computation using explicit formulae for entries of $L$, we find that the multipliers also satisfy
\begin{equation*}
\mu_1+ \mu_2 \ = \ -1- \zeta(\delta) + \eta(\delta) i
\end{equation*}
where 
\begin{equation*}
\begin{aligned}
& \zeta(\delta) = (\delta-1) (1-\delta+\delta^5+2 \delta^6-5 \delta^7+2 \delta^8\\
&\phantom{ \zeta(\delta) = (\delta-1) (1-\delta} +5 \delta^9-8\delta^{10}+4 \delta^{11}+2 \delta^{12} - 3 \delta^{13} + \delta^{14})/D_*\\
&\eta(\delta) = (\delta-1) (1-\delta+\delta^3-\delta^4+\delta^5-\delta^7+\delta^8) \times\\
&\phantom{ \zeta(\delta) = (\delta-1) (1-\delta} \times \sqrt{(1+\delta^3+\delta^6) (3-4 \delta - 4 \delta^2 + 11 \delta^3 - 4 \delta^4-4 \delta^5 + 3 \delta^6)}/D_*\\
&D_* = 2(1-\delta^2+\delta^4)(1-\delta+\delta^2-\delta^4+\delta^5)(1-\delta+\delta^3-\delta^4+\delta^5).
\end{aligned}
\end{equation*}
It follows that \chg{the two multipliers} are roots of quadratic equation
\[ t^2 +(1 +\zeta(\delta) - \eta(\delta) i ) t+ \delta=0.\]
Equivalently two multipliers are $x$-coordinates of intersection points of two curves in $\mathbb{A}^2(\C)$, $C_1 = \{ y= x^2+x+1\}$ and $C_2 = \{ y = (-\zeta(\delta) +\eta(\delta) i ) x + 1-\delta\}$. As noted above $\delta > 1$ increases to the largest real root $\delta_\infty\approx 1.68384$ of $t^7h(1/t)$ as $n\to\infty$.  For $\delta = \delta_\infty$, the intersections between $C_1$ and $C_2$ have modulus $\approx 1.69559, 0.99285$.  Since the intersections between $C_1$ and $C_2$ vary continuously with $\delta$, it follows that the $|\mu_1|<1<|\mu_2|$ for $n$ large enough.  That is, the complex fixed points are saddles.
\end{proof}

\section{Automorphisms with all periodic cycles real}
\label{sec:reallyreal}

\newcommand{\nfix}{\#\operatorname{\mathrm{Fix}}}

In the appendix to this paper, we identify some more sets of orbit data that lead to real rational surface automorphisms with maximal (and positive) entropy.  It is known (see e.g. \cite[Theorem 8.2]{can2}) that nearly all isolated periodic points have saddle type for a positive entropy complex surface automorphism $f:X\to X$.  Together with Theorem \ref{thm:allrealsaddles}, this implies that when $f$ is real and $h_{top}(f_\R) = h_{top}(f)$, nearly all periodic points of $f$ lie in $X(\R)$.  The referee for this paper asked whether it can happen that \emph{all} periodic points of $f$ are real.  This section is devoted to answering that question.

\begin{thm} 
\label{thm:245id}
The orbit data with critical orbit lengths $2,4,5$ and permutation $\sigma=\id$ is realized by a basic real map $\check f:\cp^2\to\cp^2$ that fixes $C$ with determinant $\delta >1$.  \chg{All periodic points for the corresponding rational surface autmorphism $f:X\to X$ are isolated and real, i.e. contained in $X(\R)$.  All have saddle type, except for the cusp of $C$ which is attracting.}
\end{thm}

For the remainder of this section, we fix the orbit data to be $2,4,5,\id$ as in this theorem.  Certainly some of the arguments we will use apply more generally, but we want to keep the focus on the example at hand.  The general strategy is to first rule out non-isolated fixed points for the automorphism $f:X\to X$ associated to this orbit data and then apply the Lefschetz fixed point formula \cite[Theorem 2.8.6.2]{kaha} to count and compare the numbers of fixed points for $f^n$ and $f_\R^n$.  Since the real surface $X(\R)$ is not orientable, we must first lift everything to the orientation cover $\hat\pi:\hat X \to X(\R)$ to count real fixed points.  We are grateful to Eric Bedford who described this strategy to us.  

The characteristic polynomial for the orbit data $2,4,5,\id$ \chg{can be obtained} by specializing \eqref{eqn:idcharpoly} in the appendix to this paper:
$$
\chi(t) = (1-t)^3s(t),
$$
where $s(t) = t^8-t^5-t^4-t^3+1$ has largest real root $\delta \approx 1.28064 > 1$.  

It follows from Theorem \ref{thm:realizable} and the fourth row in Table \ref{table:id} that the orbit data is realized by a real basic map $\check f:\cp^2\to\cp^2$ fixing the cuspidal cubic $C$ with determinant $\delta>1$.  We let $f:X\to X$ be the rational surface automorphism obtained by blowing up all the critical orbits of $\check f$ and continue to use $C$ to denote the anticanonical curve in $X$ obtained as the proper transform of the cuspidal cubic.  
As noted in Section \ref{sec:basic}, there is a real meromorphic two form $\omega$ on $X$ with a simple pole along $C$ and no other zeroes or poles, and $\omega$ transforms by $f$ according to $f^*\omega = \delta \omega$.

Specializing \eqref{eqn:reciprocal} and \eqref{eqn:phiid}, one finds that $(f_\R)_*:H_1(X(\R);\R)\to H_1(X(\R);\R)$ has characteristic polynomial
$$
\chi_\R(t) = (1+t^2)s(t),
$$
where $s(t)$ is as before.  Hence the characteristic polynomials for $f_*$ and $(f_\R)_*$ are the same up to cyclotomic factors, and it follows as in the proof of Theorem { \maximal} that $h_{top}(f_\R) = h_{top}(f) = \log \delta > 0$.  

Using the methods from \cite{bediki}, one can compute an explicit formula for the basic map $\check f$.  With this formula and some help from Mathematica, one locates the fixed points of $f$. Two fixed points on the invariant cubic $C$ are  $[1,1,1], [2.1003,1.2806,1]$ and two fixed points on the complement of $C$ are $[0.040129,1.2806,1],[-0.29031,0.37179,1]$.

\begin{lem}
\label{lem:twocycles}
The fixed points of $f$ consist of
\begin{itemize}
 \item an attracting fixed point at the cusp of $C$, with multipliers (i.e. eigenvalues of the derivative $Df$ at the cusp) $\delta^{-2}$ and $\delta^{-3}$;
 \item a saddle fixed point in $C_{reg}\cap X(\R)$ with multipliers $\delta$ and $\delta^{-9}$; and
 \item two fixed points in $X(\R)\setminus C$, one repelling and the other of saddle type. 
\end{itemize}
In particular, all fixed points of $f$ are real.
\end{lem}

\noindent The formulas for the multiplier at the two fixed points in C were established in much greater generality by Uehara \cite{ueh2}.  

\begin{lem}
All fixed points of $f$ are isolated.
\end{lem}

\begin{proof}
Suppose not.  Then there exists an irreducible algebraic curve $V\subset X$ and an integer $k\in \N$ such that $f^k(p) = p$ for all $p\in V$.  Since $\delta >1$, the only periodic points in $C$ are the two fixed points, so $V$ can only meet $C$ at one of these.  But $f|_V$ has finite order, whereas the multipliers at the two fixed points do not.  So $V\cap C = \emptyset$.

It follows from the genus formula (as in e.g. the proof of Theorem 3.6 in \cite{dijaso}) that $V$ is a smooth rational curve with self-intersection $-2$ and that $V$ meets any other pointwise periodic curve $V'\subset X$ for $f$ transversely.  So if
if $p\in V\cap V'$ is such a point, the multipliers of $Df^k$ at $p$ must have finite order.  This contradicts the fact that $\det Df^k(p) = \delta^k$, which holds because of the transformation property for the two form $\omega$.  Thus any two pointwise periodic curves for $f$ are disjoint.  Since they have negative self-intersection, their homology classes \chg{are} distinct.  By periodicity, the class of $V$ lies in the $f_*$-invariant subspace corresponding to the factor $(t-1)^3$ in $\chi(t)$, i.e $f_*V = V$, and so $f(V) = V$ is actually $f$-invariant.

Since $V$ has genus zero, $f|V$ is a rotation of order $k$, and there is a fixed point $p\in V$ for $f$.  By Lemma \ref{lem:twocycles}, $p\in X(\R)$.  This implies that $V = \overline{V}$ is real.  Otherwise, because $f$ is real, $\overline{V}$ is a distinct pointwise periodic curve for $f$ that meets $V$ at $p$, and we have seen that this cannot happen.  Since $V$ is smooth and real, $V\cap X(\R)$ contains no isolated points, so $V\cap X(\R)$ contains a circle $S$ invariant by $f_\R$.  The roots of $\chi_\R$ do not include $\pm 1$, so the homology class of $S$ in $H_1(X;\R)$ must be trivial.  In particular $S$ separates $X(\R)$ into two $f_\R$-invariant connected components.  One of these $U\subset X(\R)\setminus S$ is disjoint from $C(\R)$, and so (changing the sign of $\omega$ if necessary) $0 < \int_U \omega < \infty$.  This leads us to a contradiction:
$$
\int_U \omega = \int_{f_\R(U)}\omega = \int_U f^*\omega = \delta \int_U \omega.
$$
So $S$ and therefore also $V$ do not exist.
\end{proof}

The Lefschetz fixed point formula now gives that the number of fixed points of $f^n$, counted with multiplicity, is 
\begin{equation}
\label{eqn:cpxcount}
\nfix(f^n) := 2 + \mathop{\mathrm{tr}} f^n_*|H_2(X;\R) = 2 + 3 + \sum_{s(t) = 0} t^n,
\end{equation}
where the $2$ reflects the action of $f_*$ on top and bottom homology groups, and $3$ is the multiplicy of $1$ as a root of $\chi(t)$.

In order to count periodic points of $f_\R$, we must lift to the orientation cover $\hat\pi:\hat X\to X(\R)$.  Recall that $\hat X$ can be defined as the quotient by positive scaling of the determinant bundle $\det TX(\R)$.  The two points in any fiber $\hat\pi^{-1}(p)$, $p\in X(\R)$ correspond to the two possible local orientations of $X(\R)$ at $p$, and the diffeomorphism $f_\R$ lifts by pushing forward local orientations to an orientation-preserving diffeomorphism $\hat f:\hat X\to \hat X$.  We let $\iota:\hat X \to \hat X$ denote the orientation reversing diffeomorphism that swaps points in each fiber of $\hat\pi$.  

The first homology group $H_1(\hat X;\R)$ has dimension twice that of $H_1(X(\R);\R)$, and the intersection form on $H_1(\hat X;\R)$ is skew symmetric and non-degenerate.  The involution $\iota$ negates \chg{the intersection} form, i.e. $\pair{\iota_*\alpha}{\iota_*\beta} = -\pair\alpha\beta$.  Hence $H_1(X;\R) = H_+ \oplus H_-$, where the $+1$ and $-1$ eigenspaces $H_+$ and $H_-$ of $\iota_*$ are isotropic and dual to each other with respect to intersection.  In particular, $\dim H_- = \dim H_+ = \dim H_1(X(\R);\R)$.  Since $\hat\pi_*:H_1(\hat X;\R) \to H_1(X(\R);\R)$ is surjective and $\hat\pi\circ \iota = \hat\pi$, we have $\ker \hat\pi_* = H_-$ and $\hat\pi_*$ projects $H_+$ isomorphically onto $H_1(X(\R);\R)$.

\begin{lem}
For each $n\in\N$, a point $p\in \hat X$ is $n$-periodic for $\hat f$ if and only if $\hat\pi(p)$ is $n$-periodic for $f_\R$. 
\end{lem}

\begin{proof}
Since $\hat\pi:\hat X \to X$ is $2$-to-$1$, we see that $p\in \hat X$ is periodic for $\hat f$ if and only if $\hat\pi(p)$ is periodic for $f_\R$.  If the minimal period of $\hat\pi(p)$ is $n$, then the minimal period of $p$ is either $n$ or $2n$, depending on whether $f_\R^n$ is locally orientation preserving or reversing at $\hat\pi(p)$.  If $\hat\pi(p) \notin C$, then the facts that $f^*\omega = \delta\omega$ and $\delta>0$ guarantee that $f_\R$ is orientation preserving at $\hat\pi(p)$.  On the other hand, we know from Lemma \ref{lem:twocycles} that $f_\R$ is locally orientation preserving at the only two periodic (fixed) points on $C$.  In both cases, we conclude that $\hat\pi(p)$ has the same minimal period as $p$.
\end{proof}

\begin{lem}
\label{lem:doublecharpoly}
The characteristic polynomial of $\hat f_*:H_1(\hat X;\R)\to H_1(\hat X;\R)$ is $(\chi_\R(t))^2$.
\end{lem}

\begin{proof}
The map $\hat f$ commutes with $\iota$ so the eigenspaces $H_-$ and $H_+$ are $\hat f_*$-invariant.  In particular $\hat\pi_*$ conjugates the action of $\hat f_*$ on $H_+$ to that of $(f_\R)_*$ on $H_1(X(\R);\R)$, and therefore the characteristic polynomial of $\hat f_*|H_+$ is $\chi_\R(t)$.  

Pushforward $\hat f_*$ preserves the intersection form, so for any $\alpha \in H_+$, $\beta \in H_-$, we have
$$
\pair{\alpha}{\hat f_*\beta} = \pair{\hat f^{-1}_*\alpha}{\beta}.
$$
That is, the action of $\hat f_*$ on $H_-$ is congruent, and so conjugate, to that of $\hat f_*^{-1}$ on $H_+$.  The characteristic polynomial of $\hat f_*$ acting on $H_-$ is therefore given be $t^d\chi_\R(1/t)$, where $d=\deg \chi_\R(t)$.  One checks from the formulas above that $\chi_\R(t) = t^d\chi_\R(1/t)$ is a reciprocal polynomial, which finishes the proof.
\end{proof}

\begin{rem}
\label{rmk:reciprocal}
As one can see from the formulas in the appendix, the fact that $\chi_\R$ is reciprocal is apparently true in our context regardless of the orbit data.  We do not have a good non-empirical explanation for this.  For the map $f$ in this section, it can be seen as a consequence of the facts that the characteristic polynomial $\chi$ of $f_*$ is reciprocal (for intersection theoretic reasons) and that $\chi_\R$ and $\chi$ agree up to cyclotomic factors.  
\end{rem}

Now we use the fixed point formula to count fixed points of $f_\R^n$ with the help of the previous two lemmas.  It is important to remember here that, unlike periodic points of $f$, the index $\mu(f_\R^n,p)$ of an $n$-periodic point $p\in X(\R)$ for the real map $f_\R$ can be either positive or negative.  When $p$ is non-degenerate, i.e. when neither multiplier of $Df_\R^n(p)$ is $1$, we have
$$
\mu(f_\R^n,p) = \mathop{\mathrm{sign}} \det(Df^n(p)-\id) = \pm 1. 
$$
A more general formula in \cite[Theorem 2.1]{eile} implies that the real and complex multiplicities always satisfy 
$$
|\mu(f_\R^n,p)| \leq \mu(f^n,p),
$$
with equality precisely when $p$ is, again, non-degenerate.

Since $f^*\omega = \delta \omega$ we infer that $\mu(f_\R^n,p) = -1$ if and only if  $p\in X(\R)$ has saddle type with both multipliers positive.
In any case, $\mu(\hat f^n,p) = \mu(f_\R^n,\pi(p))$ for any $n$-periodic point $p\in \hat X$.

Applying the fixed point formula to $\hat f^n$ gives
$$
2\sum_{f_\R^n(p) = p} \mu(f_\R^n,p) = \sum_{\hat f^n(p)=p} \mu(\hat f^n,p) =  \sum_{j=0}^2 (-1)^j \mathop{\mathrm {tr}} \hat f^n_*|H_j(\hat X; \R) = 2 - \mathop{\mathrm{tr}} \hat f^n_*|H_1(\hat X; \R).
$$
From Lemma \ref{lem:doublecharpoly} and the formula $\chi_\R(t) = (t^2+1)s(t)$, we further have 
$$
\mathop{\mathrm{tr}} \hat f^n_*|H_1(\hat X; \R) = 4\,\re i^n + 2\sum_{s(t) = 0} t^n.
$$
For $n$ a multiple of $4$, we combine these formulas and obtain
\begin{equation}
\label{eqn:realcount}
\nfix^+(f_\R^n) - \nfix^-(f_\R^n) = \sum_{f_\R^n(p) = p} \mu(f_\R^n,p) = -1-\sum_{s(t) = 0} t^n,
\end{equation}
where $\nfix^+(f_\R^n)$ is the sum of $\mu(f_\R^n,p)$ over those $p$ where the intersection index is positive, and $\nfix^-(f_\R^n,p)$ is the magnitude of the sum over those $p$ where it is negative.

Now $\nfix(f^n) \geq \nfix^+(f_\R^n) + \nfix^-(f_\R^n)$, so we see from \eqref{eqn:cpxcount} and \eqref{eqn:realcount} that
$$
\nfix^+(f_\R^n) \leq 2 \quad\text{and}\quad \nfix^-(f_\R^n) = \nfix^+(f_\R^n) + \nfix(f^n) - 4,
$$
for all $n$ divisible by $4$.  Lemma \ref{lem:twocycles} tells us that $\nfix^+(f_\R) = 2$, so we actually have equality $\nfix(f^n) = \nfix^+(f_\R^n) + \nfix^-(f_\R^n)$ for all $n$ divisible by $4$.  In particular, all $n$-periodic points of $f$ are real.  Moreover, except for the repelling and attracting fixed points of $f$, all satisfy $\mu(f_\R^n,p) = - \mu(f^n,p)$; i.e. by the discussion above they have saddle type.

As every $n$ periodic point for $f$ is also $4n$ periodic, the Theorem is proved. 
\qed

\begin{rem}
The automorphisms associated to the orbit data in Theorem {\maximal} each have a repelling two cycle outside $X(\R)$.  One can, however, repeat the above analysis for these maps to show that all points of minimal period three and higher are saddles contained in $X(\R)$.
\end{rem}

\section{Concluding observations and questions}
\label{sec:conclusion}

In closing, we stress that this work only scratches the surface concerning the general issue of real dynamics of complex surface automorphisms.  Here we list some further problems of interest to us.  

The appendix summarizes an exhaustive case-by-case computation of characteristic polynomials for $(f_\R)_*$ associated to various sets of orbit data.
The typical situation seems to be that $(f_{\R})_*:H_1(X,\R)\to H_1(X,\R)$ has an eigenvalue outside the unit circle but none as large as the maximal eigenvalue $\delta$ of $f_*:H_2(X,\R)\to H_2(X,\R)$ for the ambient complex automorphism.  For such maps we may only conclude (using Yomdin's bound) that $h_{top}(f_{\R})$  lies in some compact subinterval of $(0,\log\delta]$.  

There are, however, some sets of orbit data beyond those described in Theorem {\maximal } that give real maps with maximal homology growth, i.e. $\rho((f_\R)_*) = \delta$, and therefore also maximal entropy.  One can check from the formulas in the appendix that these include the following.
\begin{itemize}
 \item $\sigma = \id$ is the identity, $n_1=2, n_2 =3$ and $n_3 \geq 6$;
\item $\sigma = \id$, $n_1=2, n_2 =4$ and $n_3 \geq 5$;
\item $\sigma = (12)$ is a transposition, and $n_1=1, n_2 =4$ and $n_3 \geq 6$;
\item $\sigma = (12)$, $n_1=1, n_2 =5$ and $n_3 \geq 4$;
\item $\sigma = (12)$, $n_1=1$, $n_2 \ge 8$ and $n_3 =2$.
\end{itemize}

\begin{prob}
Are there any other basic real maps as in Proposition \ref{prop:realizable} whose real dynamics have maximal entropy?
\end{prob}

\noindent We have experimentally observed, by letting one or more of the orbit lengths tend to infinity in various ways, a strong tendency for the ratio $\rho(f_*)/\rho((f_{\R})_*)$ to approach $2$ when orbit lengths become large.  This somewhat supports the idea that at least maximal homology growth is atypical for $f_\R$.

It would be natural to try to understand the real dynamics of maximal homology growth automorphisms in more detail.

\begin{prob}
Find a topological/combinatorial model for the dynamics of the real map $f_\R:X(\R)\to X(\R)$ when it has maximal homology growth.
\end{prob}

\noindent This has been done in e.g. \cite{bedi} for certain birational surface maps, but it seems harder in the present context.  

At the other extreme, there are a few cases where $f_{\R*}$ is periodic or, more generally, has spectral radius one.  These include 
\begin{itemize}
\item $\sigma=(123)$ is cyclic, $n_1=1,n_2=4,n_3=8$ : period= $180$
\item $\sigma=(123)$ is cyclic, $n_1=2,n_2=3,n_3=5$ : period= $84$
\item $\sigma=(123)$ is cyclic, $n_1=3,n_2=4,n_3=5$ : period= $126$
\item $\sigma=(123)$ is cyclic, $n_1=3,n_2=4,n_3=6$ : period= $60$
\item $\sigma=(123)$ is cyclic, $n_1=3,n_2=5,n_3=5$ : period= $168$
\item $\sigma=(123)$ is cyclic, $n_1=1,n_2=3,n_3=9$ : homology classes grow linearly under iterated pushforward. 
\end{itemize}
Lack of homology growth certainly does not imply zero entropy for diffeomorphisms of compact surfaces generally.  Pictures (e.g. {\cite[Figure~1]{can2}} and {\cite[Figure~2]{mcm2}}) of the real dynamics of automorphisms on K3 surfaces strongly suggest that lack of homology growth can coexist with positive entropy.  In these, the real surface is a sphere, i.e. simply connected, so the action $(f_\R)_*$ is automatically trivial, but orbit portraits clearly indicate complicated dynamics.  Computer pictures we have generated for our maps seem more equivocal.

\begin{prob}
Are there real rational surface automorphisms with $h_{top}(f) > 0$ but $h_{top}(f_\R) = 0$?    More generally,
are there instances in which exactly one of the inequalities is strict in the chain
$$
\log\rho((f_{\R})_*) \leq h_{top}(f_\R) \leq h_{top}(f) = \log\rho(f_*)?
$$
\end{prob}

\noindent\chg{We point out that on irrational surfaces, it can happen that only one inequality is strict.  For instance, if $f:X \to X$ is the complexification of a linear Anosov diffeomorphism on a real torus $X(\R)$, then $h_{top}(f_\R) = \log \rho((f_\R)_*) < \log \rho(f_*) = h_{top}(f)$.}

All results in this article depend heavily on restricting attention to automorphisms $f:X\to X$ that properly fix the cuspidal anticanonical curve $C$.  It is known \cite{beki2} that there are many basic maps $\check f:\cp^2\to \cp^2$ that lift to real automorphisms with positive entropy but which have no invariant curve at all.  

\begin{prob}
Find an alternative to Theorem \ref{thm:realaction} for computing the action $f_{\R*}$ on homology in the absence of an invariant anticanonical curve.
\end{prob}

Instead of considering the action of $f_\R$ on homology, one can consider the action $(f_{\R})_*:\pi_1(X(\R))\to \pi_1(X(\R))$ on the fundamental group of $X(\R)$.  If one lets $\ell(\gamma)$ denote the (minimal) word length of an element $\gamma\in\pi_1(X(\R))$ with respect to some set of generators for $\pi_1(X(\R))$, then one has \cg{\cite{bow}}
$$
h_{top}(X(\R)) \geq \limsup_{n\to\infty} \ell((f_{\R})^n_*\gamma)^{1/n}. 
$$
The right side is maximized by letting $\gamma$ range through a set of generators, and we denote the maximum by $\rho_{\pi_1}(f_\R)$, even though it is not necessarily the spectral radius of a linear operator.  
Certainly $\rho_{\pi_1}(f_\R) \geq \rho((f_{\R})_*)$.  We propose in future work to investigate

\begin{prob}
To what extent can $\rho_{\pi_1}(f_\R)$ be computed explicitly? In particular, are there examples where $h_{top}(f) = \log\rho_{\pi_1}(f_\R) > \log \rho((f_{\R})_*)$?
\end{prob}

\section*{Appendix: more general orbit data}

For the sake of completeness, we give here the analogues of Lemma \ref{lem:order} for more general orbit data $n_1,n_2,n_3,\sigma$ along with \cg{formulas for the characteristic polynomials of the actions $f_*:H_2(X;\R)\to H_2(X;\R)$ and $(f_\R)_*:H_1(X(\R);\R)\to H_1(X(\R);\R)$ on the middle homology groups.}  The characteristic polynomial for $f_*$ is known to be reciprocal, but in fact it turns out in all cases considered here that the characteristic polynomial for $(f_\R)*$ is also reciprocal of the form
\begin{equation}
\label{eqn:reciprocal}
\frac{1}{t+1} \left[ \phi(t) - (-t)^{n_1+n_2+n_3+1} \phi(1/t) \right],
\end{equation}
where $\phi$ is a polynomial that depends on the orbit data.

\begin{table}[t] 
\renewcommand{\arraystretch}{1.3}
\begin{tabular}{|c|l|}
  \hline
 Orbit Length $(n_1,n_2,n_3)$  &\phantom{AaA}   Order in the interval $ [ f_C(p_2^+)\ ,\ p_2^+ ] \subset C $\phantom{AaA}  \\
 \hline\hline
  $n_1=n_2=1$ & \hspace{2cm}$f_C(p_2^+)\ \prec\  f^{-1}_C(p_1^+) \ \prec \ f^{2}_C(p_3^+) \ =\  p_2^+  $\\ 
  \hline
  $n_1=1, n_2=2 $   &\hspace{2cm}$f_C(p_2^+)\ \prec\  f^{-1}_C(p_1^+) \ \prec \ f^{2}_C(p_3^+) \  \prec \ p_2^+$ \\ 
  \hline
  $n_1=1, n_3=2 $  &\hspace{2cm}$f_C(p_2^+)\ \prec\  f^{-1}_C(p_1^+) \ \prec \ f^{-2}_C(p_3^+) \ \prec \  p_2^+$ \\ 
  \hline
    $n_1=1, n_2+1= n_3  $  &\hspace{2cm}$f_C(p_2^+)\ \prec\  f^{-1}_C(p_1^+) \ \prec \  p_3^+ \ \prec \  p_2^+ $ \\ 
  \hline
      $n_1=1, 3\le n_2 < n_3-1 $ &\hspace{2cm}$f_C(p_2^+)\ \prec\  f_C(p_3^+) \ \prec \  f^{-1}_C(p_1^+) \ \prec \  p_2^+$ \\ 
  \hline
   $n_1=1, 3\le n_3 \le n_2-1 $ &\hspace{2cm}$f_C(p_2^+)\ \prec\    p_3^+ \ \prec \  f^{-1}_C(p_1^+) \ \prec \   p_2^+$ \\ 
  \hline
 $2 \le n_1 \le n_2 < n_3 $ &\hspace{2cm}$f_C(p_2^+)\ \prec\  f_C(p_3^+) \ \prec \  p_1^+ \ \prec \   p_2^+$ \\ 
  \hline
$2 \le n_1 \le n_3 \le n_2 $ &\hspace{2cm}$f_C(p_2^+)\ \prec\  p_1^+ \ \prec \   p_3^+ \ \prec \  p_2^+ $ \\ 
  \hline
  $n_1=n_2=n_3$, $n_1=1, n_2=n_3$ & Not realizable by a basic map properly fixing $C$ \\
 $n_1=n_2=2$, and $n_1=n_3=2$ & \\
    \hline
  \end{tabular}
\vspace{.2cm}
\caption{Cyclic Permutation, $\delta>1$, $n_1 = \min\{n_1,n_2,n_3\}$}
\label{table:cyclic}
\renewcommand{\arraystretch}{1}
\end{table}

\cg{We consider first the cyclic case $\sigma = (123)$.  Recall that the characteristic polynomial for $f_*$, given above in Equation \eqref{eqn:cycliccharpoly}, is  
$$
\chi(t) = t-t^{n_1+n_2+n_3}+(t-1)(t^{n_1}+1)(t^{n_2}+1)(t^{n_3}+1).
$$}
Table \ref{table:cyclic} gives
\begin{itemize}
\item 
all the triples $n_1,n_2,n_3$ for which the orbit data $n_1+n_2+n_3 \geq 10$ is realizable by a basic map properly fixing $C$ with $\delta > 1$; and
\item (for realizable cases) the analogue of Lemma \ref{lem:order} which determines how the critical orbits are distributed in $C$.
\end{itemize}
From this information, one arrives at the following formula for $\phi$ in \eqref{eqn:reciprocal}, valid for \emph{all} realizable cases in the table:
$$
\begin{aligned} \phi(t) =& (-1)^{n_1+n_2+n_3+1}+\frac{ (-1)^{n_2+n_3+1} \left(t^2+1\right) t^{n_1}}{t-1}\\&+\frac{ (-1)^{n_2+n_3} \left(t^3-t^2+3 t+1\right) t^{n_3}}{t^2-1}-\frac{\left(t^2+1\right) t^{n_2}}{t+1}, \qquad{n_1\leq n_2\leq n_3}
\end{aligned}
$$

\begin{table}[h]
\renewcommand{\arraystretch}{1.3}
\begin{tabular}{|c|c|}
  \hline
 Orbit Length $(n_1,n_2,n_3)$  & \phantom{Aa}  Order in the interval $ [\ f_C( p_2^+)\ ,\ p_2^+\  ] \subset C$\phantom{Aa}   \\
 \hline\hline
  $(2,3,7) $ & $f_C(p_2^+)\ \prec\  f^4_C(p_3^+) \ \prec\  f^{-2}_C(p_1^+)\ \prec\  p_2^+   $\\ 
  \hline
  $(2,3,8) $ & $f_C(p_2^+)\ \prec\ f^{-1}_C(p_1^+) \ \prec\  f^{3}_C(p_3^+)\ \prec\  p_2^+   $\\ 
    \hline
  $(2,3,n_3),\ \ n_3\ge 9$ &$f_C(p_2^+)\ \prec\  f^{3}_C(p_3^+) \ \prec\  f^{-1}_C(p_1^+)\ \prec\  p_2^+   $\\
  \hline
 $(2,4,5) $ &$f_C(p_2^+)\ \prec\ f^{-1}_C(p_1^+) \ \prec\ f_C(p_3^+)\ \prec\  p_2^+   $\\ 
    \hline
     $(2,n_2,n_3),\ \ 4 \le n_2, 6 \le n_3$ &$f_C(p_2^+)\ \prec\  f_C(p_3^+) \ \prec\  f^{-1}_C(p_1^+)\ \prec\  p_2^+   $\\
  \hline
    $(n_1,n_2,n_3), \ 3 \le n_1 \lneqq n_2 \lneqq n_3 $  &$f_C(p_2^+)\ \prec\ f_C(p_3^+) \ \prec\  p_1^+\ \prec\  p_2^+   $\\
  \hline
  $n_i = n_j$ for $i\ne j$ & Degenerate Case\\
  \hline
\end{tabular}
\vspace{.2cm}
\caption{$\sigma=Id$,  $\delta>1$}\label{table:id}
\renewcommand{\arraystretch}{1}
\end{table}

\begin{table}[h]
\renewcommand{\arraystretch}{1.3}
\begin{tabular}{|c|c|}
  \hline
 Orbit Length $(n_1,n_2,n_3)$  & \phantom{Aa}  Order in the interval $ [\ f_C( p_2^+)\ ,\ p_2^+\  ]\subset C $\phantom{Aa}   \\
 \hline\hline
  $(1,8,2) $ & $f_C(p_2^+)\ \prec\  f^{-4}_C(p_3^+)   \ \prec\  f^{-2}_C(p_1^+)\ \prec\  p_2^+$\\ 
  \hline
 $(1,n_2,2),\ \ n_2\ge 9 $ & $f_C(p_2^+)\ \prec\  f^{-3}_C(p_3^+)   \ \prec\  f^{-2}_C(p_1^+)\ \prec\  p_2^+$\\ 
  \hline
 $(1,4,6) $ & $f_C(p_2^+)\ \prec\  f^{4}_C(p_3^+)   \ \prec\  f^{-1}_C(p_1^+)\ \prec\  p_2^+$\\ 
  \hline
 $(1,4,n_3),\ \ n_3\ge 7 $ & $f_C(p_2^+)\ \prec\  f^{3}_C(p_3^+)   \ \prec\  f^{-1}_C(p_1^+)\ \prec\  p_2^+$\\ 
  \hline
 $(1,n_2,n_3), n_3 \ge n_2-1 \ge 4 $ & $f_C(p_2^+)\ \prec\  f_C(p_3^+)   \ \prec\  f^{-1}_C(p_1^+)\ \prec\  p_2^+$\\ 
  \hline
 $(1,n_2,n_3),\ \ 4\le n_3 \le n_2-3, n_2\ge 5 $ & $f_C(p_2^+)\ \prec\  p_3^+   \ \prec\  f^{-1}_C(p_1^+)\ \prec\  p_2^+$\\ 
  \hline
     $2\le n_1 \lneqq n_2  \le n_3 , $ & $f_C(p_2^+)\ \prec \  f_C(p_3^+)  \ \prec\  p_1^+ \ \prec\  p_2^+$\\ 
  \hline
   $2\le n_1 \lneqq n_3 \lneqq n_2 $ & $f_C(p_2^+)  \ \prec\  p_1^+\ \prec \   p_3^+\ \prec\  p_2^+$\\ 
   \hline
      $3 \le n_3 \le n_1 \lneqq n_2  $ & $f_C(p_2^+)\ \prec \  p_3^+  \ \prec\  p_1^+ \ \prec\  p_2^+$\\ 
  \hline
  
     $(n_1,n_2,2),  n_1 \ge 3 $ & \\
 $ \ (n_1,n_2) \ne (3,6),(3,7),(3,8),(4,5) $ & $f_C(p_2^+)\ \prec \  f^{-1}_C(p_3^+)  \ \prec\  p_1^+ \ \prec\  p_2^+$\\ 
  \hline

  $(2,n_2,2),\ \ n_2\ge 8 $ & $f_C(p_2^+) \ \prec \  f^{-2}_C(p_3^+) \ \prec\  f^{-1}_C(p_1^+)\ \prec\  p_2^+$\\ 
    \hline
  
   $(2,3,6) $ & $f_C(p_2^+)  \ \prec\  p_1^+\ \prec \  f^{3}_C(p_3^+)\ \prec\  p_2^+$\\ 
  \hline
  $(2,3,7) $ & $f_C(p_2^+) \ \prec \  f^{3}_C(p_3^+) \ \prec\  p_1^+\ \prec\  p_2^+$\\ 
  \hline
  $(2,3,n_3),\ \ n_3\ge 8  $ & $f_C(p_2^+)  \ \prec\  p_1^+\ \prec \  f^{2}_C(p_3^+)\ \prec\  p_2^+$\\ 
  \hline
     $(2,7,2) $ & $f_C(p_2^+) \ \prec \  f^{-3}_C(p_3^+) \ \prec\  f^{-1}_C(p_1^+)\ \prec\  p_2^+$\\ 
  \hline
  
   $n_3=2,  \ (n_1,n_2) = (3,6),(3,7),(3,8),(4,5)$ & $f_C(p_2^+)  \ \prec\  p_1^+\ \prec \   f^{-2}_C(p_3^+)\ \prec\  p_2^+$\\ 
  \hline
$ (1, n_2, 3)$ & $f_C(p_2^+) =p_3^+ \ \prec\    f^{-1}_C(p_1^+)\ \prec\  p_2^+$\\ 
\hline
$(1, n_2, n_2-2)$ & $f_C(p_2^+) \ \prec\    f^{-1}_C(p_1^+)=p_3^+ \ \prec\  p_2^+$\\ 
\hline
   $n_1=n_2$ &  Not realizable by a basic map properly fixing $C$ \\
  \hline
\end{tabular}
\vspace{.2cm}
\caption{$\sigma=(1,2)$ and $\delta>1$}\label{table:transpo2}
\renewcommand{\arraystretch}{1}
\end{table}

\begin{table}[h]
\renewcommand{\arraystretch}{1.3}
\begin{tabular}{|c|c|}
  \hline
Orbit Length  $n_1 < n_2$  & $\phi(t)$   \\
 \hline\hline

  $n_1=3, n_2=4, n_3=3$ & $1-t^3-t^4+t^5$ \\
  \hline
  $n_1=1, n_2=4, n_3 \ge 6$ &$1-t-t^2+2 t^3-t^4$\\
  \hline
  $n_1=1, n_2 \ge5 ,n_3 \ge n_2-1$ & $t^n + ( -1-t^3+ 2 t^{n-2})/(t-1)$\\
  \hline 
  $n_1=2,n_2=3, n_3\ge 6 $ & $1-t^2+t^4-t^5$\\
  \hline 
  $n_3=2< n_2-1$ & $ 1+t^3 - t^{n_1} -t^{1+n_1}+t^{2+n_1}-t^{3+n_1}$ \\
  \hline
  $n_1=n_3 < n_2-1 $ & $1+ (-1)^{n_1} t^{n_1}+ t^{2 n_1} (t-1) $\\&$+ 2 (-1)^{n_1} t^{n_1} ((-t)^{n_1}+t)/(t+1)$\\
  \hline
  $2 < n_3 \le n_1-1 $ & $1+ (-1)^{n_3} ( t^{n_1}+t^{n_1+1}+t^{n_3+1}) +(t-1) t^{n_1+n_3}$\\& $- 2 (-1)^{n_3}  t^{n_1+2} (1-(-t)^{n_3-2})/(t+1)$\\ 
  \hline 
  $ n_1+1 \le n_3 < n_2-1$ & $1+ (-1)^{n_1}( t^{1+n_3}-t^{n_1} (1+t))+t^{n_1+n_3} (t-1) $\\&$+ 2 (-1)^{n_1+1}t^{2+n_3} (1-(-t)^{n_1-2})/(t+1) $\\ &$ + 2 (-1)^{n_1+1}(t^{1+n_3}-t^{2+n_1})/(t-1)$\\
  \hline 
  $ n_3 \ge n_2-1$ & $1+ (-1)^{n_1} t^{n_2} (1-t) + t^{n_1+n_2} + (-1)^{1+n_1} t^{n_1} (1+t) $\\
  &$+2 (-1)^ {1+n_1} (-t^{2+n_1}+t^{n_2})/(t-1) $\\
  \hline

\end{tabular}
\vspace{.2cm}
\caption{$\sigma=(1,2)$ }\label{table:transpo3}
\renewcommand{\arraystretch}{1}
\end{table}

Now suppose $\sigma$ is the identity permutation.  \cg{The characteristic polynomial for $f_*$ is (see Equation (2) in \cite{dil})
\begin{equation}
\label{eqn:idcharpoly}
\chi(t) =(t-1) (t^{n_1+n_2+n_3}-t^{n_1}-t^{n_2}-t^{n_3}+2) - (t^{n_1}-1) (t^{n_2}-1) (t^{n_3}-1).
\end{equation}} 
Table \ref{table:id} sums up the situation for $f_\R$.  Again, case-by-case computation leads to formulas
\begin{equation}
\label{eqn:phiid}
\begin{aligned}
\phi(t) &\ =\ 1-2 t+3 t^3-3 t^4+t^5, & \text{if }\ \ n_1=2, n_2=3 \\
\phi(t) &\ =\  1+ (-1)^{n_1} t^{n_1+1} - 2 t^{n_2} + (-1)^{1+n_1} t^{1+n_2} + t^{n_1+n_2}, &\qquad{n_1\leq n_2\leq n_3}\\
\end{aligned}
\end{equation}
for the polynomial $\phi(t)$ in \eqref{eqn:reciprocal}.

Finally we consider the case when $\sigma = (12)$ is a transposition.  \cg{The characteristic polynomial for $f_*$ is 
$$
\chi(t) =(t-1) (t^{n_3} (t^{n_1}+1) (t^{n_2}+1) - t^{n_1} - t^{n_2}-2) - (t^{n_1+n_2}-1) (t^{n_3}-1).
$$}
The situation for $f_\R$ is more complicated than before.  Results are listed in Table  \ref{table:transpo2}.  In the transposition case, we have not found a single formula for the polynomial $\phi$ that gives the characteristic polynomial \eqref{eqn:reciprocal}, so we list the possibilities case-by-case in Table \ref{table:transpo3}. \cg{ For the last three cases in Table  \ref{table:transpo2}, the parameters $t_i$ in Theorem \ref{thm:realizable} are not distinct, so Theorem \ref{thm:realizable} does not guarantee that there actually is a basic map that realizes the given orbit data.}

%


\end{document}